\documentclass[12pt,review,
]{elsarticle}

\usepackage{amsmath,amsthm,amssymb}
\usepackage{euscript,mathrsfs}
\usepackage{verbatim}

\usepackage{geometry}
\usepackage{hyperref}
	\hypersetup{pdfauthor=author}
\usepackage[capitalize]{cleveref}
	\crefname{equation}{}{}
	\crefformat{section}{\textsection#2#1#3}
	\crefformat{appendix}{\textsection#2#1#3}

\theoremstyle{plain}
	\newtheorem{thm}{Theorem}[section]
	\newtheorem{lem}[thm]{Lemma}
	\newtheorem{pro}[thm]{Proposition}
	\newtheorem{cor}[thm]{Corollary}
	
	\newtheorem{thmk}{Theorem}

\theoremstyle{definition}
	\newtheorem{defn}[thm]{Definition}
	\newtheorem{rem}[thm]{Remark}
	\newtheorem{example}[thm]{Example}
	\newtheorem*{notation}{Notation}
	\newenvironment{eg}{\begin{footnotesize} \begin{example}}{\end{example}\end{footnotesize}}

\newcommand{\N}{\mathbb{N}}
\newcommand{\Z}{\mathbb{Z}}
\newcommand{\R}{\mathbb{R}}
\newcommand{\C}{\mathbb{C}}
\newcommand{\T}{\mathbb{T}}
\newcommand{\cB}{\mathcal{B}}
\newcommand{\cH}{\mathcal{H}}
\newcommand{\cC}{\mathcal{C}}
\newcommand{\eG}{\EuScript{G}}
\newcommand{\eF}{\EuScript{F}}
\newcommand{\eS}{\EuScript{S}}
\newcommand{\eU}{\EuScript{U}}
\newcommand{\fS}{\mathfrak{S}}

\newcommand{\tr}{\mathrm{tr\,}}
\newcommand{\rank}{\mathrm{rank\,}}
\newcommand{\dist}{\mathrm{dist}}
\newcommand{\supp}{\mathrm{supp\,}}
\newcommand{\flow}{\mathrm{sf}\,}
\newcommand{\dis}{\sigma_{\mathrm{dis}}}
\newcommand{\ess}{\sigma_{\mathrm{ess}}}

\newcommand{\Lip}{{\normalfont \textbf{Lip\textsubscript{*}}}}
\newcommand{\SP}{\mathrm{SP}}

\newcommand{\textbi}[1]{\textit{\textbf{#1}}}

\begin{document}

\begin{frontmatter}

\title{A Topological Approach to Unitary Spectral Flow via Continuous Enumeration of Eigenvalues} 

\author[UNSW]{Nurulla Azamov}
\ead{nurulla.azamov@unsw.edu.au}

\author[Independent Scholar]{Tom Daniels}
\ead{tomdaniels86@hotmail.com }

\author[Shinshu]{Yohei Tanaka\corref{corresponding}}
\ead{20hs602a@shinshu-u.ac.jp}

\cortext[corresponding]{Corresponding author}

\address[UNSW]{University of New South Wales, Kensington, NSW, 2052, Australia}
\address[Independent Scholar]{Independent scholar, SA, Australia}
\address[Shinshu]{Division of Mathematics and Physics, Faculty of Engineering, Shinshu University, Wakasato, Nagano 380-8553, Japan}

\begin{abstract}
It is a well-known result of T.\,Kato that given a continuous one-parameter family of square matrices of a fixed dimension, the eigenvalues of the family can be chosen continuously. In this paper, we give an infinite-dimensional analogue of this result, which arises in the context of unitary spectral flow. This intuitive topological approach to unitary spectral flow via continuous enumeration of eigenvalues appears to be missing from the existing literature, and it is the purpose of the present paper to fill in the gap. It is also shown in this paper that the notion of continuous enumeration naturally leads to a variant of the celebrated theorem of Dold-Thom in algebraic topology.
\end{abstract}

\begin{keyword}
Unitary Spectral Flow \sep Continuous enumeration of Eigenvalues \sep Schatten-class Perturbations \sep Dold-Thom Theorem
\end{keyword}

\end{frontmatter}

\section{Introduction}
\label{Section: Introduction}

Let us first start with finite-dimensional continuous enumeration due to T.\,Kato. The following exposition is directly taken from \cite[\textsection VI.1]{Book:Bhatia97}. Let $X$ be a metric space endowed with a metric $d,$ and let $\SP^n(X)$ be the finite $n$-th symmetric product of $X.$ Recall that $\SP^n(X)$ is the quotient topological space obtained from $X^n$ via the equivalence relation which identifies two $n$-tuples of elements, if they are permutations of each other (see \cite[\textsection 3.C]{Book:Hatcher02} for details). That is, $\SP^n(X)$ can be viewed as the space of unordered $n$-tuples of elements of $X.$ We denote by $[\lambda_1,\dots,\lambda_n]$ the equivalence class represented by an $n$-tuple $(\lambda_1,\dots,\lambda_n) \in X^n.$ The symmetric product $\SP^n(X)$ is metrisable by
\[
\dist\,([\lambda_1,\dots,\lambda_n],[\lambda'_1,\dots,\lambda'_n])
    := \min_{\pi} \max_{1 \leq i \leq n} d(\lambda_i, \lambda'_{\pi_i}),
\]
where the minimum is taken over all permutations $\pi.$ The following result is well-known;

\begin{thmk}[{\cite[Theorem II.5.2]{Book:Kato95}}]
\label{Theorem: Kato's Selection Theorem}
Let $I$ be any subinterval of $[-\infty,\infty],$ and let $\lambda$ be a continuous $\SP^n(\C)$-valued mapping on $I.$ Then there exist finitely many continuous functions $\lambda_1,\dots,\lambda_n : I \to \C,$ such that $\lambda(t) = [\lambda_1(t),\dots,\lambda_n(t)]$ for all $t \in
I.$
\end{thmk}
As is typical, a selection theorem of this kind is not altogether straightforward to prove even in this finite-dimensional setting. On one hand the domain $I$ cannot be extended to contain an open subset of the complex plane (see, for example, \cite[Example VI.1.3]{Book:Bhatia97}). On the other hand the range $\SP^n(\C)$ can be replaced by a general finite symmetric product $\SP^n(X).$ Note that the continuity of the following mapping is also well-known (see, for example, \cite[\textsection VI.1]{Book:Bhatia97});
\begin{equation}
\label{Equation: Bhatia's Mapping}
M_n(\C) \ni A \longmapsto [\lambda_1(A),\dots,\lambda_n(A)] \in \SP^n(\C),
\end{equation}
where $M_n(\C)$ is the set of all $n \times n$ matrices over $\C$ and $\lambda_1(A),\dots,\lambda_n(A)$ are the eigenvalues of $A$ repeated according to their multiplicities. We may identify the unordered tuple $[\lambda_1(A),\dots,\lambda_n(A)]$ on the right hand side of \cref{Equation: Bhatia's Mapping} with the spectrum $\sigma(A)$ of the matrix $A.$ The following result, referred to as \textit{Kato's finite-dimensional continuous enumeration of eigenvalues}, is an immediate consequence of \cref{Theorem: Kato's Selection Theorem} and the continuity of \cref{Equation: Bhatia's Mapping};

\begin{thmk}[Kato's finite-dimensional continuous enumeration of eigenvalues]
\label{Theorem: Kato's Finite-dimensional Continuous Enumeration}
Let $I$ be any subinterval of $[-\infty,\infty],$ and let $A$ be a continuous $M_n(\C)$-valued mapping on $I.$ Then there exist finitely many continuous functions  $\lambda_1, \dots, \lambda_n : I \to \C,$ such that $\sigma(A(t)) = [\lambda_1(t), \dots, \lambda_n(t)]$ for all $t \in I.$
\end{thmk}

This paper gives a certain infinite-dimensional analogue of \cref{Theorem: Kato's Finite-dimensional Continuous Enumeration} which naturally arises in the context of \textit{unitary spectral flow}. In the setting of operators on Hilbert space, the scope of continuous enumeration is obviously restricted to discrete eigenvalues, or pure point spectrum, and does not make sense for continuous components of spectrum. We will consider only the discrete spectrum, consisting of isolated eigenvalues of finite multiplicity, and only compact perturbations thereof. This ensures that there are countably many eigenvalues of interest, which can accumulate at the essential spectrum, whose location, given Weyl's theorem, is fixed. For concreteness, let us consider the topological group $\eU_\Phi(\cH,1)$ of all those unitary operators on $\cH$ with the property that $U - 1$ belongs to the \textit{Schatten class} $\fS_\Phi(\cH)$ of compact operators determined by the \textit{symmetric norm} $\Phi$ (see \cref{Section: Symmetric Norms} and \cref{Section: Preliminaries of Flow of Discrete Spectrum} respectively for the definitions of $\Phi$ and $\fS_\Phi(\cH)$). Note that the topology on $\eU_\Phi(\cH, 1)$ is given by the following complete metric;
\[
\dist_\Phi(U,U') := \|U - U'\|_{\Phi}, \qquad U,U' \in \eU_\Phi(\cH,1),
\]
where $\|\cdot\|_{\Phi}$ is the standard norm on $\fS_\Phi(\cH).$ It follows from Weyl's theorem that any unitary operator $U \in \eU_\Phi(\cH,1)$ shares the common essential spectrum $\ess(U) = \{1\}$ on the unit-circle $\T.$ The following infinite-dimensional analogue of \cref{Theorem: Kato's Finite-dimensional Continuous Enumeration} is one of the main theorems of this paper;

\begin{thm}
\label{Theorem: Baby Continuous Enumeration}
Let $I$ be any subinterval of $[-\infty,\infty],$ and let $\{U(t)\}_{t \in I}$ be a continuous one-parameter family of operators in $\eU_\Phi(\cH,1).$ Then there exist infinitely many continuous functions $\lambda_1, \lambda_2,  \dots : I \to \T,$ such that for each $t \in I$ the sequence $(\lambda_j(t))_{j \in \N}$ contains all of the members of the spectrum $\sigma(U(t)),$ each of which is repeated according to its multiplicity. Here, the multiplicity of $1$ is defined to be infinite.
\end{thm}

Motivated by Kato's finite-dimensional approach, we shall break the proof of \cref{Theorem: Baby Continuous Enumeration} into two parts. Firstly, we introduce a certain infinite analogue of the finite symmetric product $\SP^n(\T),$ which will be denoted in this case by $\eS_\Phi(\T,1),$ so that each $\sigma(U(t))$ can be naturally viewed as a member of $\eS_\Phi(\T,1).$ 
We define a metric on $\eS_\Phi(\T,1)$ in such a way that an infinite-dimensional variant of the Hoffman-Wielandt inequality for normal matrices (see \cref{Section: Preliminaries of Flow of Discrete Spectrum} for details) immediately implies the Lipschitz continuity of the family $\{\sigma(U(t))\}_{t\in I}.$ Secondly, we prove the existence of a continuous enumeration for any continuous path in $\eS_\Phi(\T,1),$ the statement of which is nothing but \cref{Theorem: Kato's Selection Theorem} with $\SP^n(\C)$ replaced by $\eS_\Phi(\T,1)$ and the finitely many continuous functions $\lambda_1, \dots, \lambda_n$ replaced by infinitely many continuous functions. 

The origin of \textit{spectral flow} for continuous one-parameter families of self-adjoint Fredholm operators goes back to \cite{Atiyah-Patodi-Singer75}. Given such a family $\{F(t)\}_{t \in [0,1]},$ we can understand its spectral flow as the net number of eigenvalues of $F(t)$ that cross $0$ rightward. This well-known homotopy invariant has since found many connections, for example, to the Fredholm index \cite{Robbin-Salamon95}. The main focus of the present paper is the notion of spectral flow for unitary operators. As with the self-adjoint Fredholm case, we can intuitively understand the spectral flow of a continuous one-parameter family $\{U(t)\}_{t \in [0,1]}$ in $\eU_\Phi(\cH,1)$ to be the integer-valued function $\flow(-;\{U(t)\}_{t \in [0,1]}) : (0, 2\pi) \to \Z$ given by
\begin{align}
\notag 
\flow(\theta;\{U(t)&\}_{t \in [0,1]}) :=  \\
\label{Equation: Naive Definition of Unitary Spectral Flow} \langle &\mbox{the number of eigenvalues of $U(t)$ that cross $e^{i\theta}$ anticlockwise} \rangle \\
- \langle &\mbox{the number of eigenvalues of $U(t)$ that cross
$e^{i\theta}$ clockwise} \rangle \nonumber
\end{align}
as $t$ monotonically increases from $0$ to $1.$ 
In \cite{Pushnitski01} the definition \cref{Equation: Naive Definition of Unitary Spectral Flow} is made precise, and is used to express the spectral shift function (see \cite{Lifshitz52,Krein53,Gesztesy-Makarov00,Book:Yafaev92,Simon98}) as the averaged spectral flow of a certain continuous path of unitary operators connecting the scattering matrix with the identity in $\eU_{\Phi}(\cH,1)$. 

\cref{Theorem: Baby Continuous Enumeration} provides another intuitive and direct method for showing that \cref{Equation: Naive Definition of Unitary Spectral Flow} is well-defined. 
This approach is used in \cite{Azamov11,Azamov-Daniels19} to express the absolutely continuous part of the spectral shift function as the averaged spectral flow of a path in $\eU_{\Phi}(\cH,1)$, which connects the scattering matrix with the identity in a homotopically non-equivalent way to that of~\cite{Pushnitski01}, and therefore to express the singular part of the spectral shift function as the integer spectral flow of a loop in $\eU_{\Phi}(\cH,1)$ based at the identity.
Note, however, that many non-trivial results including the existence of continuous enumeration appear in these papers without proofs, and those missing proofs can be found in the present paper. 

For simplicity, let us assume that the given family $\{U(t)\}_{t \in [0,1]}$ is a loop of the form $U(0) = U(1) = 1.$ In this case, the spectral flow of $\{U(t)\}_{t \in [0,1]}$ does not depend on the choice of the reference point $e^{i \theta},$ since it represents the net number of the windings that the eigenvalues of $\{U(t)\}_{t \in [0,1]}$ make in the anti-clockwise direction. The family $\{\sigma(U(t))\}_{t \in [0,1]}$ admits a continuous enumeration $\lambda_1, \lambda_2, \dots$ in the sense of \cref{Theorem: Baby Continuous Enumeration}, and this allows us to define 
\begin{equation}
\label{Equation: Naive Definition of Spectral Flow for Loops}
\flow(\{U(t)\}_{t \in [0,1]}) := [\lambda_1]_{\pi_1} +  [\lambda_2]_{\pi_1} +  [\lambda_3]_{\pi_1} + \dots,
\end{equation}
where each homotopy class $[\lambda_j]_{\pi_1}$ represents the winding number of the loop $\lambda_j$ in the fundamental group $\pi_1(\T,1) = \Z.$ Note that \cref{Equation: Naive Definition of Spectral Flow for Loops} provides an explicit group homomorphism from $\pi_1(\eU_\Phi(\cH,1),1)$ into $\Z,$ thereby proving that the subgroup $\eU_\Phi(\cH,1)$ of the unitary group $\eU(\cH)$ is not simply connected, even though $\pi_1(\eU(\cH))$ is well-known to be trivial. 

For full generality, we shall consider an arbitrary based metric space $(X,x_0)$ instead of $(\T,1).$ This paper is organised as follows. The purpose of \cref{Section: Summable Multisets} is to introduce the metric space $\eS_\Phi(X,x_0),$ whose members are infinite unordered tuples of elements of $X,$ each of which contains the basepoint $x_0$ infinitely many times and satisfies a certain summability condition with respect to $\Phi.$ In \cref{Section: Continuous Enumeration of Multiset-valued Mappings} we prove \cref{Theorem: Continuous Enumeration}, the existence of continuous enumeration for continuous paths in $\eS_\Phi(X,x_0).$ 
This allows us to give an intuitive exposition of the unitary spectral flow in \cref{Section: Flow of Discrete Spectrum}. 
This paper concludes with Appendix \cref{Section: Appendix}, in which the functor $(X,x_0) \longmapsto \eS_\Phi(X,x_0)$ is shown to preserve both separability and completeness. This supplementary material can be read independently from \cref{Section: Continuous Enumeration of Multiset-valued Mappings} and \cref{Section: Flow of Discrete Spectrum}.

On a final note, the formula \cref{Equation: Naive Definition of Spectral Flow for Loops} motivates us to introduce a concrete isomorphism from the fundamental group $\pi_1(\eS_\Phi(X,x_0))$ onto the first singular homology group $H_1(X)$ via continuous enumeration. The authors would like to thank D.\,Tamaki for pointing out that this is an analogue of the \textit{Dold-Thom theorem}, which states $\pi_1(\SP^\infty(X,x_0)) \simeq H_1(X)$ (see, for example, \cite[\textsection 4.K]{Book:Hatcher02}). The rigorous treatment of this material can be found in Y.T.'s master's thesis \cite{Tanaka14}, and it will be the subject of another paper in preparation.

\section{Summable Multisets}
\label{Section: Summable Multisets}

The current section is organised as follows. In \cref{Section: Symmetric Norms} we given an overview of symmetric norms. The reader who is not familiar with this notion may as well assume that the symmetric norm $\Phi$ in this paper is any \textbi{$p$-norm} $\Phi_p,$ where $1 \leq p \leq \infty,$ and directly proceed to \cref{Section: Countable Multisets}. Recall that to each real-valued sequence $\xi = (\xi_i)_{i \in \N}$ we assign
\begin{equation}
\label{Equation: Definition of p-norm}
\Phi_p(\xi) :=
    \begin{cases}
        \left(\sum^\infty_{i=1} |\xi_i|^p \right)^{1/p}, & \mbox{if } p < \infty, \\
        \sup_{i \in \N} |\xi_i|, & \mbox{if } p = \infty.
    \end{cases}
\end{equation}
Note that the Banach space $\ell^p(\N) := \ell_{\Phi_p}(\N)$ consists of all those real-valued sequences with finite $p$-norm, and that $\Phi_p$ has the regularity property in the sense of \cref{Lemma: Characteristion of Regularity} (ii). In \cref{Theorem: Metric Space Theorem} we introduce the metric space $\eS_\Phi(X,x_0),$ where $(X,x_0)$ is a based metric space and $\Phi$ is a symmetric norm. We topologise $\eS_\Phi(X,x_0)$ in such a way that each finite symmetric product $\SP^n(X)$ (see \cref{Section: Introduction} for details) can be continuously embedded into $\eS_\Phi(X,x_0)$ in a canonical fashion. The rest of the current section is a summary of  preliminary results for the proceeding sections. The separability and completeness of $\eS_\Phi(X,x_0)$ will be discussed in Appendix \cref{Section: Appendix}. This supplementary material can be read independently from \cref{Section: Continuous Enumeration of Multiset-valued Mappings} and \cref{Section: Flow of Discrete Spectrum}.

\subsection{Symetric norms}
\label{Section: Symmetric Norms}

Here, we briefly recall standard facts about symmetric norms for the reader's convenience (details can be found in \cite[\textsection III.3]{Book:Gohberg-Krein69} and \cite[\textsection 1.1.7]{Book:Simon05}). Let $c_0$ be the set of all real-valued sequences converging to $0,$ and let $c_{00}$ be the set of all real-valued sequences with only finitely many non-zero terms. Evidently, $c_0$ and $c_{00}$ can be both viewed as vector spaces over $\R.$ 

\begin{defn}
Any norm $\Phi$ on $c_{00},$ which assigns to each sequence $\xi = (\xi_i)_{i \in \N}$ in $c_{00}$ a unique non-negative number $\Phi(\xi) = \Phi(\xi_1,\xi_2,\dots),$ is called a \textbi{symmetric norm}, if the following two conditions hold true:
\begin{enumerate}[(i)]
    \item We have $\Phi(1,0,0,\dots) = 1.$
    \item We have $\Phi(\xi_1,\xi_2,\dots) = \Phi(|\xi_{\pi_1}|,|\xi_{\pi_2}|,\dots)$ for any $\xi \in c_{00}$ and any permutation $\pi.$
\end{enumerate}
Given such $\Phi,$ a sequence $\xi \in c_0$ is said to be \textbi{$\Phi$-summable}, if the following limit exists;
\begin{equation}
\label{Equation: Extension of Symmetric Norms}
\Phi(\xi) := \lim_{i \to \infty} \Phi(\xi_1,\dots,\xi_i,0,0,\dots).
\end{equation}
\end{defn}

The set $\ell_\Phi(\R)$ of all $\Phi$-summable sequences forms a Banach space with respect to the norm \cref{Equation: Extension of Symmetric Norms} (see, for example, \cite[Theorem 1.16]{Book:Simon05}). We have the following obvious assertion;

\begin{lem}
\label{Lemma: Characteristion of Regularity}
If $\Phi$ is a symmetric norm, then the following conditions are equivalent:
\begin{enumerate}[(i)]
\item The vector space $c_{00}$ is a dense subspace of $\ell_\Phi(\R).$
\item For each $\xi \in \ell_\Phi(\R),$ we have $\Phi(\xi_{i+1},\xi_{i+2},\dots) \to 0$ as $i \to \infty.$
\item For each $\xi \in \ell_\Phi(\R),$ we have $(\xi_1,\dots,\xi_i,0,0,\dots) \to \xi$ as $i \to \infty.$
\end{enumerate}
\end{lem}
The symmetric norm $\Phi$ is said to be \textbi{regular} (or \textbi{mononormalising} as in \cite[\textsection III.6]{Book:Gohberg-Krein69}), if the above equivalent conditions hold true.  

\begin{eg}
It is shown in \cite[\textsection III.7]{Book:Gohberg-Krein69} that the $p$-norm $\Phi_p$ given by \cref{Equation: Definition of p-norm} is a regular symmetric norm for each $p \in [1,\infty].$ On the other hand, the so-called \textbi{Calderon norms} are examples of symmetric norms that are not regular (see, for example, \cite[\textsection 1.1.7]{Book:Simon05}). 
\end{eg}

\begin{lem}[{\cite[Lemma III.3.1]{Book:Gohberg-Krein69}}]
\label{Lemma: Majorisation}
Let $\Phi$ be a symmetric norm, and let $\xi,\eta \in \ell_{\Phi}(\R)$ satisfy $|\xi_i| \geq |\xi_{i+1}|$ and $|\eta_i| \geq |\eta_{i+1}|$ for each $i \in \N.$ If $\sum^k_{i=1} |\xi_i| \leq \sum^k_{i=1} |\eta_i|$ for all $k \in \N,$ then $\Phi(\xi) \leq \Phi(\eta).$
\end{lem}

We shall make use of the following corollary throughout this paper;
\begin{cor}
\label{Corollary: Basic Properties of Symmetric Norms}
Let $\Phi$ be a symmetric norm, and let $\xi,\eta \in \ell_\Phi(\R).$ Then the following assertions hold true:
\begin{enumerate}[(i)]
\item We have $\Phi(\xi_1,\xi_2,\dots) = \Phi(|\xi_{\pi_1}|,|\xi_{\pi_2}|,\dots)$ for any permutation $\pi.$ 
\item If $|\xi_i| \leq |\eta_i|$ for each $i \in \N,$ then $\Phi(\xi) \leq \Phi(\eta).$
\item We have $\sup_{i \in \N} |\xi_i| \leq \Phi(\xi) \leq \sum^\infty_{i=1} |\xi_i|.$
\end{enumerate}
\end{cor}
Note that the last assertion implies $\ell^1(\R) \subseteq \ell_\Phi(\R) \subseteq \ell^\infty(\R).$
\begin{proof}
For the first assertion, note that $\Phi(\xi) = \Phi(|\xi|)$ immediately follows from the fact that $\Phi$ is symmetric, and so it is sufficient to show that $\Phi(\xi) = \Phi(\xi_{\pi})$ for each permutation $\pi.$ There exists an increasing sequence $(N_n)_{n \in \N}$ of natural numbers, such that $\xi_{\pi_1}, \dots, \xi_{\pi_n}$ are among $\xi_{1}, \dots, \xi_{{N_n}}.$ It follows from \cref{Lemma: Majorisation} that 
\[
\Phi(\xi_{\pi_1}, \dots, \xi_{\pi_n},0,0,\dots) \leq \Phi(\xi_{1}, \dots, \xi_{{N_n}},0,0,\dots), \qquad n \in \N.
\]
Taking the limit as $n \to \infty$ gives $\Phi(\xi_\pi) \leq \Phi(\xi).$ Since $\pi$ was chosen arbitrarily, the claim follows (observe $\xi = \xi_{\pi \circ \pi^{-1}}$). The remaining assertions are also easy consequences of the same lemma; see \cite[\textsection III.3]{Book:Gohberg-Krein69} for details.
\end{proof}

\subsection{Countable multisets}
\label{Section: Countable Multisets}
Let $X$ be a set with a basepoint $x_0 \in X.$ A \textbi{multisubset} of $X$ is understood naively as a subset of $X,$ whose elements can be repeated more than once. For instance, the multisubset $\{x,x\}^*,$ where we use notation $\{\dots\}^*$ to distinguish it from ordinary subsets of $X,$ is considered to be  different from $\{x\}^*.$ Let
\begin{equation}
\label{Equation: Definition of Zero Multiset}
O_{x_0} := \{x_0,x_0,x_0,\dots\}^*,
\end{equation}
where $x_0$ is repeated infinitely many times. More precisely,

\begin{defn}
A \textbi{multisubset} of $X$ is any mapping of the form $S : X \to \{0,1,2,\dots,\infty\}$ assigning to each point $x \in X$ a unique non-negative integer or infinity, $S(x),$ which is defined to be the \textbi{multiplicity} of $x$ in $S.$
\end{defn}

A \textbi{countable multisubset of $(X,x_0)$} is any multisubset $S$ of $X,$ such that the basepoint $x_0$ is the only point in $S$ having infinite multiplicity, and the \textbi{support} of $S$ defined as follows is countable;
\begin{equation}
\supp S := \{x \in X \mid S(x) > 0\}.
\end{equation}
The \textbi{rank} of $S,$ denoted by $\rank S,$ is the sum of the multiplicities of all points in $\supp S$ except the basepoint $x_0.$ We call $S$ a \textbi{finite-rank multisubset}, if $\rank S < \infty.$ Evidently, the multisubset $O_{x_0}$ given by \cref{Equation: Definition of Zero Multiset} is the only finite-rank multisubset of $(X,x_0)$ with zero rank. Given a subset $U$ of $X,$ we agree to write $S \subseteq U,$ if $\supp S \subseteq U.$ The notation $x \in S,$ where $x \in X,$ is understood likewise as $x \in \supp S.$

We will only consider countable multisubsets of $(X,x_0)$ from here on, and so the following convention makes sense. Whenever we are given a multisubset $\{s_1,s_2,\dots\}^*$ of $X,$ where the sequence $(s_1,s_2, \dots)$ may be finite or infinite, we shall \textit{always} assume that $\{s_1,s_2,\dots\}^*$ contains the basepoint $x_0$ infinitely many times. With this convention in mind, we introduce the following terminology;

\begin{defn}
Let $S$ be a countable multisubset of $(X,x_0).$ A sequence $(s_i)_{i \in \N}$ is called an \textbi{enumeration} of $S,$ if $S = \{s_1,s_2,\dots\}^*.$ If the given enumeration $(s_i)_{i \in \N}$ contains the basepoint $x_0$ infinitely many times, then it is called a \textbi{proper enumeration} of $S.$
\end{defn}

\begin{rem}
\label{Remark: Proper Enumeration}
Let $S$ be a countable multisubset of $(X,x_0).$ Any two proper enumerations of $S$ are identical up to a permutation. Furthermore, given an enumeration $(s_i)_{i \in \N}$ of $S,$ the sequence $(s_1,x_0,s_2,x_0,\dots)$ is a proper enumeration of $S.$
\end{rem}

Given two countable multisubsets $S,T$ of $(X,x_0),$ we agree to write $T \subseteq S,$ if $T(x) \leq S(x)$ for all $x \in X.$ We define the \textbi{sum} $S + T,$ and \textbi{difference} $S - T$ in case $T
\leq S,$ by
\[
(S \pm T)(x) =
    \begin{cases}
    \infty, & \mbox{if $x = x_0,$}\\
    S(x) \pm T(x), & \mbox{otherwise.}
    \end{cases}
\]
Given a countable multisubset $S$ of $(X,x_0)$ and an arbitrary subset $U$ of $X,$ we define their \textbi{intersection}, denoted by $S \cap U,$ as follows;
\[
(S \cap U)(x) :=
    \begin{cases}
        \infty, & \mbox{if $x = x_0,$} \\
        S(x), & \mbox{if $x \neq x_0$ and $x \in U,$}\\
        0, & \mbox{if $x \neq x_0$ and $x \notin U.$}
    \end{cases}
\]
Note that the multiplicity of the basepoint $x_0$ in $S \cap U$ is defined to be infinite, even if the basepoint $x_0$ does not belong to the set $U.$ Thus, we can always view the intersection $S \cap U$ as a countable multisubset of $(X,x_0).$ We also define
$
S \setminus U := S \cap (X \setminus U).
$

\subsection{Summable multisets}
\begin{notation}
Throughout the remaining part of the current section, let $\Phi$ be a symmetric norm, and let $(X,x_0)$ be a based metric space endowed with a fixed metric $d.$ Let $B_\epsilon(x) := \{x' \in X \mid d(x,x') < \epsilon\}$ be the open $\epsilon$-neighborhood of a fixed point $x \in X.$ 
\end{notation}

Given two countable multisubsets $S,T$ of $(X,x_0),$ we define their \textbi{$\Phi$-distance} by
\begin{equation}
\label{Equation: Phi-distance for Multisets}
d_\Phi(S,T) := \inf \; \Phi(d(s_1,t_1),d(s_2,t_2),\dots),
\end{equation}
where the infimum is taken over all pairs $(s_i)_{i \in \N},(t_i)_{i \in \N}$ of enumerations of $S,T$ respectively. In fact, we may assume without loss of generality that the infimum in \cref{Equation: Phi-distance for Multisets} is taken over all \textbi{proper} enumerations $(s_i)_{i \in \N},(t_i)_{i \in \N}$ of $S,T$ respectively (\cref{Remark: Proper Enumeration}).

\begin{pro}
\label{Theorem: Metric Space Theorem}
The set $\eS_\Phi(X,x_0)$ of all those countable multisubsets $S$ of $(X,x_0)$ with $d_\Phi(O_{x_0}, S) < \infty$ forms a metric space with respect to the distance function \cref{Equation: Phi-distance for Multisets}.
\end{pro}
The members of the metric space $\eS_\Phi(X,x_0)$ are said to be \textbi{$\Phi$-summable}. Note that $S \in \eS_\Phi(X,x_0)$ if and only if given any enumeration $(s_i)_{i \in \N}$ of $S$ we have $(d(x_0,s_i))_{i \in \N} \in \ell_\Phi(\R).$ To prove \cref{Theorem: Metric Space Theorem} let us first prove the following lemma;

\begin{lem}
\label{Lemma: Metric Space Remark}
If $S,T$ are multisubsets in $\eS_\Phi(X,x_0)$ admitting enumerations $(s_i)_{i \in \N},(t_i)_{i \in \N}$ respectively, then the following assertions hold true:
\begin{enumerate}[(i)]
\item The support of $S$ is a compact set, and it can have only one accumulation point $x_0.$
\item We have
\begin{equation}
\label{Equation: Between Trace Norm and Sup Norm}
\sup_{i \in \N} d(s_i,t_i) \leq \Phi(d(s_1,t_1),d(s_2,t_2),\dots) \leq \sum^\infty_{i=1} d(s_i, t_i).
\end{equation}
\item We have $\Phi(d(x_0,s_{i+1}),d(x_0,s_{i+2}),\dots) \to 0$ as $i \to \infty,$ provided that the given symmetric norm $\Phi$ is regular in the sense of \cref{Lemma: Characteristion of Regularity}.
\end{enumerate}
\end{lem}
\begin{proof}
The first assertion follows from the fact that $\supp S$ is the closure of the image of a sequence converging to $x_0.$ Since $(d(s_i,t_i))_{i \in \N} \in \ell_\Phi(\R),$ the remaining assertions follow from \cref{Corollary: Basic Properties of Symmetric Norms} (iii) and \cref{Lemma: Characteristion of Regularity} (ii).
\end{proof}

\begin{proof}[Proof of \cref{Theorem: Metric Space Theorem}]
Let $S,T,U \in \eS_\Phi(X,x_0).$ We shall make use of the first two assertions in \cref{Lemma: Metric Space Remark}. Note that $d_\Phi(S,T) = d_\Phi(T,S)$ and $d_\Phi(S,S) = 0$ are obvious. As for non-degeneracy, we assume the contrary that $d_\Phi(S,T) = 0$ with $S \neq T.$ Without loss of generality, we may assume that there exists a point $x' \neq x_0,$ such that $S(x') < T(x').$ Since $x'$ cannot be an accumulation point of $\supp S,$ we can choose a small enough open $\epsilon$-ball $B_\epsilon(x')$ around $x',$ such that $B_\epsilon(x') \cap \supp S$ is either the empty set $\emptyset$ or the singleton $\{x'\}.$  In either case, this leads to a contradiction $d_\Phi(S,T) \geq \epsilon > 0.$

To prove the triangle inequality $d_\Phi(S,T) \leq d_\Phi(S,U) + d_\Phi(U,T),$ we let $(s_i)_{i \in \N},(t_i)_{i \in \N}$ be proper enumerations of $S,T$ respectively, and let $(u_i)_{i \in \N},(u'_i)_{i \in \N}$ be two proper enumerations of $U.$ Then there exists a permutation $\pi$ satisfying $u'_{\pi_i} = u_i$ for each $i \in \N.$
\begin{align*}
\Phi[(d(s_i,u_i))_{i \in \N}] + \Phi[(d(u'_i,t_i))_{i \in \N}] 
    &= \Phi[(d(s_i,u_i))_{i \in \N}] + \Phi[(d(u_i,t_{\pi_i}))_{i \in \N}] \\
    &\geq \Phi[(d(s_i,u_i))_{i \in \N} + (d(u_i,t_{\pi_i}))_{i \in \N}] \\
    & \geq \Phi[(d(s_i,t_{\pi_i}))_{i \in \N}] \\
    &\geq d_\Phi(S,T).
\end{align*}
Since all the proper enumerations $(s_i)_{i \in \N},(t_i)_{i \in \N},(u_i)_{i \in \N},(u'_i)_{i \in \N}$ were chosen arbitrarily, taking the infimum over these sequences establishes the triangle inequality. In particular, selecting $U := O_{x_0}$ ensures $d_\Phi(S,T) < \infty$ for all $S,T \in \eS_\Phi(X,x_0).$
\end{proof}

\begin{rem}
We let $\eS_p(X,x_0) := \eS_{\Phi_p}(X,x_0)$ and $d_p := d_{\Phi_p}$ for each fixed $p \in [0,\infty],$ where the $p$-norm $\Phi_p$ is given by \cref{Equation: Definition of p-norm}. It immediately follows from \cref{Equation: Between Trace Norm and Sup Norm} that we have
\begin{equation}
\label{Equation: Between Trace Multiset and Sup Multiset}
\eS_1(X,x_0) \subseteq \eS_\Phi(X,x_0) \subseteq  \eS_\infty(X,x_0),
\end{equation}
where each inclusion is $1$-Lipschitz continuous.
\end{rem}

\subsection{Multiset functors}
\label{Section: Multiset functors}
Let $(Y,y_0)$ be another based metric space, and let $f : (X,x_0) \to (Y,y_0)$ be a based $L$-Lipschitz continuous mapping. That is, $f$ is an  $L$-Lipschitz continuous mapping of the form $f : X \to Y$ with $f(x_0) = y_0.$ It is easy to see that $f$ naturally induces the based $L$-Lipschitz mapping $f_* : (\eS_\Phi(X,x_0),O_{x_0}) \to (\eS_\Phi(Y,y_0),O_{y_0})$ defined by
\begin{equation}
\label{Equation: Induced Lipschitz}
\eS_\Phi(X,x_0) \ni \{s_1,s_2,\dots\}^* \longmapsto \{f(s_1),f(s_2),\dots\}^* \in \eS_\Phi(Y,y_0).
\end{equation}

More precisely, we consider the category $\Lip$ consisting of based metric spaces as objects and based Lipschitz mappings as morphisms. For each fixed symmetric norm $\Phi,$ the \textbi{$\Phi$-multiset functor} assigns to each based metric space $(X,x_0)$ a new based metric space $(\eS_\Phi(X,x_0),O_{x_0})$ and to each morphism $f : (X,x_0) \to (Y,y_0)$ a new morphism $f_* : (\eS_\Phi(X,x_0),O_{x_0}) \to (\eS_\Phi(Y,y_0),O_{y_0}).$ This defines a covariant functor of the form $\Lip \to \Lip.$

\subsection{Metric inequalities}
The purpose of the current subsection is to introduce certain inequalities about $d_\Phi,$ which will be used throughout this paper.
\subsubsection{Inequalities involving sum and difference}

\begin{lem}\label{Sum Inequality}
We have $d_\Phi(S + S', T + T') \leq d_\Phi(S,T) + d_\Phi(S',T')$ for all
$S,S',T,T' \in \eS_\Phi(X,x_0).$
\end{lem}
\begin{proof}
Let $(s_i)_{i \in \N},(s'_i)_{i \in \N},(t_i)_{i \in \N},(t'_i)_{i \in \N}$ be enumerations of $S,S',T,T'$
respectively. Since
$(s_1,s'_1,s_2,s'_2,s_3,s'_3,\dots)$ and $(t_1,t'_1,t_2,t'_2,t_3,
t'_3, \dots)$
are enumerations of $S + S', T + T'$ respectively, we have
\begin{align*}
\Phi[(d(s_i,t_i))_{i \in \N}] &+ \Phi[(d(s'_i,t'_i))_{i \in \N}] \\
&= \Phi(d(s_1,t_1),0,d(s_2,t_2),0,\dots) + \Phi(0,d(s'_1,t'_1),0,d(s'_2,t'_2),\dots) \\
&\geq \Phi(d(s_1,t_1),d(s'_1,t'_1),d(s_2,t_2),d(s'_2,t'_2),\dots) \\
&\geq d_\Phi(S + S',T + T'),
\end{align*}
where the first inequality follows from the triangle inequality with respect to $\Phi.$ Since the enumerations $(s_i)_{i \in \N},(s'_i)_{i \in \N},(t_i)_{i \in \N},(t'_i)_{i \in \N}$ were chosen arbitrarily, taking the infimum over these enumerations establishes the claim.
\end{proof}

An analogous inequality $d_\Phi(S - S', T - T') \leq d_\Phi(S,T) + d_\Phi(S',T'),$ where $S,S',T,T' \in  \eS_\Phi(X,x_0)$ satisfy $S' \subseteq S$ and $T' \subseteq T,$ fails to hold in general as in the following example;

\begin{eg}
Let $N > 1$ be a fixed natural number. Here, we consider the space $\eS_2(\R_+,0),$ where the set $\R_+$ of non-negative real numbers is equipped with the standard metric $\rho(x,y) := |x - y|.$ We define finite-rank multisubsets $S,S',T,T'$ in $\eS_2(\R_+,0)$ by
\[
S = S' = T = \left\{\frac{1}{N},\frac{2}{N},\dots,1\right\}^*,  \qquad T'= T - \{1\}^* = \left\{\frac{1}{N},\frac{2}{N},\dots,\frac{N-1}{N}\right\}^*.
\]
Since
$\left(\frac{1}{N},\frac{2}{N},\dots,1, 0, 0, 0,\dots\right),
\left(0,\frac{1}{N},\frac{2}{N},\dots,\frac{N-1}{N}, 0, 0,0,
\dots\right)$ are enumerations of $S',T'$ respectively,
\[
\rho_2(S,T) + \rho_2(S',T')
    \leq 0 + \left(\left|\frac{1}{N} - 0\right|^2 + \left|\frac{2}{N} - \frac{1}{N}\right|^2 + \dots + \left|1 - \frac{N-1}{N} \right|^2\right)^{\frac{1}{2}} 
    \leq \frac{1}{\sqrt{N}} < 1.
\]
On the other hand, since $S  -S' = O_0$ and $T - T' = \{1\}^*,$ we have $\rho_2(S-S',T-T') = \rho_2(O_0,\{1\}^*) = 1.$ That is, $d_\Phi(S - S', T - T') \leq d_\Phi(S,T) + d_\Phi(S',T')$ does not hold true in general.
\end{eg}

Nevertheless, the following weaker version turns out to be sufficient;

\begin{lem}\label{Difference Inequality 2}
Let $S,S',T,T' \in \eS_\Phi(X,x_0)$ be multisubsets satisfying $S' \subseteq S$
and $T' \subseteq T.$ If $S',T'$ are finite-rank multisubsets and if $n = \rank (S' + T'),$ then
\[
d_\Phi(S-S',T-T') \leq2^n \left(d_\Phi(S,T) +
d_\Phi(S',T')\right).
\]

\end{lem}

This result will be proved with the aid of the following lemma;
\begin{lem}\label{Lemma1: Difference Inequality}
If $S,T,U \in \eS_\Phi(X,x_0)$ and if $n = \rank U < \infty,$ then
\[
d_\Phi(S,T) \leq 2^n d_\Phi(S + U,T + U).
\]
\end{lem}
\begin{proof}
Let us first prove the claim for $U = \{u\}^*.$ Let $(s'_i)_{i \in \N},(t'_i)_{i \in \N}$ be enumerations of $S + U,T + U$ respectively, such that $s'_{i_0} = u$ and $t'_{j_0} = u$ for some $i_0, j_0 \in \N.$ Note first that if $i_0 = j_0,$ then $d_\Phi(S,T) \leq  \Phi[(d(s'_i,t'_i))_{i \in \N}] \leq 2 \Phi[(d(s'_i,t'_i))_{i \in \N}]$ trivially holds true. On the other hand, if $i_0 \neq j_0,$ we can then simultaneously renumber $(s'_i)_{i \in \N},(t'_i)_{i \in \N},$ so that
$(s'_i)_{i \in \N} = (s_1,u,s_2,s_3,\dots)$ and $(t'_i)_{i \in \N} = (u,t_1,t_2,t_3,\dots)$ for some enumerations $(s_i)_{i \in \N},(t_i)_{i \in \N}$ of $S,T$ respectively. It follows that
\begin{align*}
d_\Phi(S,T)
& \leq \Phi(d(s_1,t_1),d(s_2,t_2), \dots) \\
& \leq \Phi(d(s_1,u) + d(u,t_1), d(s_2,t_2), d(s_3,t_3), \dots) \\
& \leq \Phi(d(s_1,u) + d(u,t_1), 2d(s_2,t_2), 2d(s_3,t_3),\dots) \\
& \leq \Phi(d(s_1,u), d(s_2,t_2), \dots) + \Phi(d(u,t_1), d(s_2,t_2), \dots) \\
& \leq 2\Phi(d(s_1,u), d(u,t_1), d(s_2,t_2), d(s_3,t_3), \dots) \\
& \leq 2\Phi(d(s'_1,t'_1), d(s'_2,t'_2),  d(s'_3,t'_3), \dots).
\end{align*}
Taking the infimum over $(s'_i)_{i \in \N}, (t'_i)_{i \in \N}$  establishes the claim for $U = \{u\}^*.$ For the general case, suppose $U = \{u_1,\dots,u_n\}^*.$ It follows that
\[
2^n  d_\Phi(S + U,T + U) \geq 2^{n-1}  d_\Phi(S + \{u_1,\dots,u_{n-1}\}^*,T + \{u_1,\dots,u_{n-1}\}^* ).
\]
Continuing this way, we obtain the claim.
\end{proof}

\begin{proof}[Proof of \cref{Difference Inequality 2}]
Since $\rank(S' + T') = n,$ we get
\begin{align*}
d_\Phi(S-S',T-T')
    &\leq 2^n  d_\Phi(S-S'+(S'+T'),T - T' + (S'+T')) \\
    &\leq 2^n  d_\Phi(S + T',T + S') \\
    &\leq 2^n \left(d_\Phi(S,T) + d_\Phi(S',T')\right),
\end{align*}
where the first inequality follows from \cref{Lemma1: Difference Inequality} and the last inequality follows from \cref{Sum Inequality}. The proof is complete.
\end{proof}

\subsubsection{Inequalities involving finite-rank multisets}

\begin{lem}
Given $s_0,s_1,\dots,s_n \in X$ and $t_1,\dots,t_n \in X,$ we have
\begin{align}
\label{Equation: Finite Sum} &d_\Phi(\{s_1,\dots,s_n\}^*,\{t_1,\dots,t_n\}^*) \leq \sum^n_{i=1} d(s_i,t_i), \\
\label{Equation: Phi-estimate} &\sup_{1 \leq i \leq n} d(s_0,s_i)
    \leq 2\, d_\Phi(\{\underbrace{s_0,\dots,s_0}_{\mbox{$n$ times}}\}^*,\{s_1,\dots,s_n\}^*).
\end{align}
\end{lem}
\begin{proof}
Inequality \cref{Equation: Finite Sum} follows from \cref{Sum Inequality} and the second inequality in \cref{Equation: Between Trace Norm and Sup Norm};
\[
d_\Phi(\{s_1,\dots,s_n\}^*,\{t_1,\dots,t_n\}^*)
    \leq d_\Phi(\{s_1\}^*,\{t_1\}^*) + \dots + d_\Phi(\{s_n\}^*,\{t_n\}^*) \\
    \leq \sum^n_{i=1} d(s_i,t_i).
\]
Inequality \cref{Equation: Phi-estimate} follows from the first inequality in \cref{Equation: Between Trace Norm and Sup Norm} and the triangle inequality with respect to $d.$
\end{proof}

\subsubsection{Inequalities involving intersection}

Given $S,T \in \eS_\Phi(X,x_0)$ and a subset $U$ of $X,$ the following inequality does not hold true in general;
\begin{equation}
\label{Intersection Estimate}
d_\Phi(S \cap U, T \cap U) \leq d_\Phi(S, T).
\end{equation}
Here, we establish a criterion under which estimate \cref{Intersection Estimate} holds true.

\begin{defn}
\label{Definition: Definition of Positively Separated Tuples}
We introduce the following terminology:
\begin{enumerate}[(i)]
\item Given a subset $U$ of $X,$ we let $\eS_\Phi^U(X,x_0) := \{S \in \eS_\Phi(X,x_0) \mid S \subseteq U\}.$
\item A finite tuple $(U_0, \dots, U_k)$ of non-empty subsets of $X$ is \textbi{positively separated}, if
\begin{align}
&x_0 \in U_0, \\
&\dist(U_i, U_j) := \inf_{(u_i,u_j) \in U_i \times U_j} d(u_i,u_j) > 0, \qquad i \neq j.
\end{align}
\end{enumerate}
The positive number $\delta := \min_{i \neq j} \dist(U_i,U_j)$ is called the \textbi{separation} of $(U_0, \dots, U_k).$
\end{defn}

\begin{lem}
\label{Lemma: Intersection Inequality}
Let $(U_0,\dots,U_k)$ be a positively separated tuple of subsets of $X$ with separation $\delta > 0.$ If $U = \bigcup_{j=0}^k U_j$ and if $S,T \in \eS_\Phi^U(X,x_0)$ with
$d_\Phi(S,T) < \delta,$ then:
\begin{enumerate}[(i)]
\item We have $d_\Phi(S \cap U_j, T \cap U_j) \leq d_\Phi(S,T)$ for all $j= 0, \dots, k.$
\item We have $\rank (S \cap U_j) = \rank (T \cap U_j)$ for all $j = 1, \dots, k.$
\end{enumerate}
\end{lem}
Note that $\rank (S \cap U_0) = \rank (T \cap U_0)$ does not hold true  in general.
\begin{proof}
Suppose that $S,T \in \eS_\Phi^U(X,x_0)$ satisfy $d_\Phi(S,T) < \delta,$ and that $(s_i)_{i \in \N},(t_i)_{i \in \N}$ are arbitrary enumerations of $S,T$ respectively. Without loss of generality, we may assume that $\Phi[(d(s_i,t_i))_{i \in \N}] < \delta.$ It follows from the first inequality in \cref{Equation: Between Trace Norm and Sup Norm} that
\begin{equation}
\label{Equation : Intersection Inequality}
\sup_{i \in \N} d(s_i,t_i) < \delta.
\end{equation}
It follows that each set $U_j$ has the property that $s_i \in U_j$ if and only if  $t_i \in U_j$ for all $i \in \N.$ The second assertion follows. The first assertion follows from $d_\Phi(S \cap U_j, T \cap U_j) \leq \Phi[(d(d_i,t_i))_{i \in \N}]$ for each $j = 0, \dots, n.$
\end{proof}

\begin{cor}
\label{Corolary: Intersection Continuity}
If $(U_0, \dots, U_k)$ is a positively separated tuple of subsets of $X$ and if $U = \bigcup_{j=0}^k U_j,$ then the following mappings are continuous:
\begin{equation}
\label{Equation: Intersection Continuity}
\eS_\Phi^U(X,x_0) \ni S \longmapsto S \cap U_j \in \eS_\Phi(X,x_0), \qquad  j = 0, \dots, k.
\end{equation}
\end{cor}
\begin{proof}
This immediately follows from \cref{Lemma: Intersection Inequality}.
\end{proof}

\begin{lem}
\label{Lemma: Openness of Restricted Sp}
If $U$ is an open neighborhood of $x_0,$ then $\eS_\Phi^U(X,x_0)$ is open in $\eS_\Phi(X,x_0).$
\end{lem}
We are interested in the case $U = U_0 \cup \dots \cup U_k,$ where $(U_0,\dots,U_k)$ forms a positively separated tuple of \textit{open} subsets of $X.$
\begin{proof}
Let $S = \{s_1,s_2,\dots\}^*$ be a multiset in $\eS_\Phi^U(X,x_0),$ and let $\delta := \dist(\supp S, X \setminus U).$ Since the compact set $\supp S$ and the closed set $X \setminus U$ are disjoint, we have $\delta > 0.$ If $T \in \eS_\Phi(X,x_0)$ is a multisubset satisfying $d_\Phi(S,T) < \delta,$ then there exists an enumeration $(t_i)_{i \in\N}$ of $T,$ such that \cref{Equation : Intersection Inequality} holds true. It follows that $t_i \in U$ for all $i \in \N,$ and so $T \in \eS_\Phi^U(X,x_0).$ Thus, $\eS^U_\Phi(X,x_0)$ is an open subset of $\eS_\Phi(X,x_0).$
\end{proof}

\subsection{Continuity of multiset-valued mappings}

The purpose of the current subsection is to establish several results about continuity of multiset-valued mappings. We will
make use of the following terminology;
\begin{defn}
The \textbi{rank} of a $\eS_\Phi(X,x_0)$-valued mapping $S$ on a set $I$ is defined to be the smallest non-negative number $n$ such that $\rank S(t) \leq n$ for all $t \in I.$ The mapping $S$ is called a \textbi{finite-rank mapping}, if it has a finite rank.
\end{defn}

\begin{notation}
Let $I$ be a fixed metric space.
\end{notation}

\subsubsection{Continuity of sum, difference, and intersection}

Let $S,T : I \to \eS_\Phi(X,x_0)$ be two mappings, and let $U$ be a subset of $X.$ We define $S + T : I \to \eS_\Phi(X,x_0)$ and $S - T : I \to \eS_\Phi(X,x_0)$ in case of $T(t) \subseteq S(t)$ for all $t \in I,$ by $S \pm T := S(\cdot) \pm T(\cdot).$ We also define $S \cap U: I \to \eS_\Phi(X,x_0)$ by $S \cap U := S(\cdot) \cap U.$ 

\begin{pro}
\label{Proposition: Continuity of Sum and Difference}
If $S,T : I \to \eS_\Phi(X,x_0)$ are continuous mappings, then we have the following assertions:
\begin{enumerate}[(i)]
\item The mapping $S + T : I \to \eS_\Phi(X,x_0)$ is continuous.
\item If $T(t) \subseteq S(t)$ for all $t \in I$ and if a point $t_0 \in I$ has a neighborhood $I_0$ such that the restriction $T|_{I_0}$ is finite-rank, then $S - T : I \to \eS_\Phi(X,x_0)$ is continuous at $t_0.$
\end{enumerate}
\end{pro}
\begin{proof}
It follows from \cref{Sum Inequality} that
\[
d_\Phi((S+T)(t),(S+T)(t')) \leq d_\Phi(S(t),S(t')) + d_\Phi(T(t),T(t')),  
\qquad t,t' \in I.
\]
The continuity of $S + T$ follows from that of $S,T.$ As for the continuity of the difference $S - T,$ assume that $T(t) \subseteq S(t)$ for all $t \in I,$ and that a point $t_0 \in I$ has a neighborhood $I_0$ such that the restriction $T|_{I_0}$ is finite-rank. That is, for some large enough $n \in \N$ we have $\rank(T(t) + T(t_0)) \leq n$ for each $t \in I_0.$ It follows from \cref{Difference Inequality 2} that
\[
d_\Phi((S-T)(t_0),(S-T)(t)) \leq 2^{n} \left(d_\Phi(S(t_0),S(t)) + d_\Phi(T(t_0),T(t))\right), \qquad t \in I_0.
\]
The continuity of $S - T$ at $t_0$ follows from that of $S,T.$
\end{proof}

\begin{pro}
\label{Proposition: Reducing Open Sets}
Let $S : I \to \eS_\Phi(X,x_0)$ be a continuous mapping, and let $t_0 \in I$ be fixed. Suppose $\supp S(t_0) \subseteq U_0 \cup \dots \cup U_k$ for some positively separated tuple $(U_0, \dots, U_k)$ of open subsets of $X.$ Then there exists a neighborhood $I_0$ of $t_0$ such that the following mappings are all continuous:
\begin{equation}
\label{Equation: Reducing Open Sets and Intersection}
I_0 \ni t \longmapsto S(t) \cap U_i \in \eS_\Phi(X,x_0), \qquad i=0,1, \dots, k.
\end{equation}
Furthermore, the neighborhood $I_0$ has the following properties:
\begin{enumerate}[(i)]
\item We have $S(t) = S(t) \cap U_0 + \dots + S(t) \cap U_k$ for all $t \in I_0.$
\item We have $d_\Phi(S(t) \cap U_i, S(t') \cap U_i) \leq d_\Phi(S(t),S(t'))$ for each $t,t' \in I_0$ and each $i=0,1, \dots, k.$
\item We have $\rank (S(t) \cap U_i) = \rank (S(t') \cap U_i)$ for each $t,t' \in I_0$ and each $i= 1, \dots, k.$
\end{enumerate}
\end{pro}
\begin{proof}
Let $U = \bigcup_{j=0}^k U_j.$ It follows from the continuity of $S$ and \cref{Lemma: Openness of Restricted Sp} that there exists a  neighborhood $I_0$ of $t_0$ such that $S(t) \in \eS_\Phi^U(X,x_0)$ for each $t \in I_0.$ The continuity of \cref{Equation: Reducing Open Sets and Intersection} follows from the fact that it is the composition of $I_0 \ni t \longmapsto S(t) \in \eS_\Phi^U(X,x_0)$ and \cref{Equation: Intersection Continuity}. We may assume without loss of generality that for each $t,t' \in I_0$ the metric $d_\Phi(S(t),S(t_0))$ never exceeds the separation $\delta$ of the tuple $(U_0, \dots, U_k).$ The remaining assertions immediately follow from \cref{Lemma: Intersection Inequality}.
\end{proof}

\subsubsection{Continuity of induced mappings}
Note that finitely many mappings $\lambda_1, \dots, \lambda_n : I \to X$ induce the mapping $I \ni t \longmapsto \{\lambda_1(t),\dots, \lambda_n(t)\}^* \in \eS_\Phi(X,x_0),$ which will be denoted by $\{\lambda_1, \dots, \lambda_n\}^*.$ If the mappings $\lambda_1, \dots, \lambda_n$ are all continuous, then so is the induced mapping $\{\lambda_1, \dots, \lambda_n\}^*$ by \cref{Equation: Finite Sum}. To give an infinite-rank version of this result, we introduce the following terminology;

\begin{defn}
\label{Pointwise Summable}
A sequence $(\lambda_i)_{i \in \N}$ of $X$-valued mappings on the metric space $I$ is said to be \textbi{pointwise $\Phi$-summable}, if for each $t \in I$ we have $\{\lambda_1(t), \lambda_2(t),\dots\}^* \in \eS_\Phi(X,x_0).$
\end{defn}

As with the finite-rank case above, given a pointwise $\Phi$-summable sequence $(\lambda_i)_{i \in \N}$ of continuous mappings on $I,$ we denote by $\{\lambda_1,\lambda_2, \dots\}^*$ the following mapping;
\begin{equation}
\label{Equation: Induced Mapping}
I \ni t \longmapsto \{\lambda_1(t),\lambda_2(t),\dots\}^* \in \eS_\Phi(X,x_0).
\end{equation}
The following example shows that \cref{Equation: Induced Mapping} is \textit{not} continuous in general;

\begin{eg}
\label{Example: Epsilon Separation Does Not Hold for Compact Intervals}
We shall consider the metric space $\eS_1(\R_+,0),$ where the set $\R_+$ of non-negative real numbers is equipped with the standard metric $\rho(x,y) := |x - y|.$ For each $j \in \N,$ let $x_j := 1/2^{j-1},$ and let $I_j := [x_{j+1},x_{j}].$ Suppose that we have a sequence $(\lambda_j)_{j \in \N}$ of $\R_+$ -valued continuous functions, where each $\lambda_j : I_j \to \R_+$ satisfies the following three conditions:
\[
\lambda_j(x_{j+1}) = 0, \qquad \lambda_j \left(\frac{x_{j+1} + x_j}{2}\right) = 1, \qquad \lambda_j(x_j) = x_j.
\]
We can then continuously extend the domain $I_j = [x_{j+1},x_{j}]$ of each $\lambda_j$ to $[0,1]$ as follows;
\begin{equation}
\lambda_j(t) = 
\begin{cases}
x_j, & \mbox{if } x_j \leq t, \\
\lambda_j(t), & \mbox{if } t \in I_j, \\
0, & \mbox{if } t \leq x_{j+1}. 
\end{cases}
\end{equation}
Let $S(t) := \{\lambda_1(t),\lambda_2(t),\lambda_3(t), \dots\}^*$ for each $t \in [0,1].$ The extended sequence $(\lambda_j)_{j \in \N}$ is pointwise summable, since for each $j \in \N$ and each $t \in I_j$ we have
\[ 
S(t) = \{0, \dots, 0, \lambda_{j}(t), x_{j+1}, x_{j+2}, \dots\}^*, \qquad
S(0) = O_0, \qquad S(1) = \{x_1, x_2, x_3, \dots \}^*.
\]
It remains to show that the mapping $[0,1] \ni t \longmapsto S(t) \in \eS_1(\R_+,0)$ fails to be continuous at $t=0.$ Indeed, any open neighborhood of $0$ contains some $I_j \ni (x_j + x_{j+1})/2,$ and so we get
\begin{equation}
\label{Equation: Epsilon Separation Does Not Hold for Compact Intervals}
\rho_1(S(t_j), O_0) = \lambda_{j}(t_j) + \sum_{j=1}^\infty x_{n+1} \geq \lambda_{j}(t_j) = 1, \qquad t_j := \frac{x_{j+1} + x_{j}}{2}.
\end{equation}
On the other hand, the restriction of $S$ to the non-compact interval $(0,1] := \bigcup_{j \in \N} I_j$ is continuous, since $S$ on each $I_j$ is continuous. Indeed, we have $S(t) = \{\lambda_{j}(t)\}^* + \{x_{j+1}, x_{j+2}, \dots\}^*$ for each $t \in I_j.$
\end{eg}

Nevertheless, we have the following criterion;
\begin{pro}
\label{Theorem: Induce Criterion}
Let $\Phi$ be regular in the sense of \cref{Lemma: Characteristion of Regularity}, and let $I$ be compact. Let $(\lambda_i)_{i \in \N}$ be a pointwise $\Phi$-summable sequence of continuous $X$-valued mappings on $I.$ Then the following statements are equivalent:
\begin{enumerate}[(i)]
\item We have $\Phi(d(x_0,\lambda_{i+1}(\cdot)), d(x_0,\lambda_{i+2}(\cdot)), \dots) \to 0$ uniformly as $i \to \infty.$
\item The mapping $\xi_0 := (d(x_0,\lambda_i(\cdot)))_{i \in \N} : I \to \ell_\Phi(\R)$ is continuous.
\item The induced mapping $S := \{\lambda_1(\cdot),\lambda_2(\cdot),\dots\}^*$ is continuous.
\end{enumerate}
\end{pro}
Note that in \cref{Example: Epsilon Separation Does Not Hold for Compact Intervals} the sequence $\xi_0 := (\rho(0, \lambda_{i}(\cdot)))_{i \in \N} = (\lambda_{i}(\cdot))_{i \in \N}$  fails to be continuous at $0$ due to \cref{Equation: Epsilon Separation Does Not Hold for Compact Intervals}. As we shall see shortly, the assumption of $I$ being compact will be only used in the implication $\textnormal{(iii)} \Rightarrow \textnormal{(i)}.$
\begin{proof}
We proceed as $\textnormal{(i)} \Rightarrow \textnormal{(ii)} \Rightarrow \textnormal{(iii)} \Rightarrow \textnormal{(i)}.$  Before taking up the proof, let us first introduce some notation. For each $i \in \N,$ we define
\[
S_i(\cdot) := \{\lambda_1(\cdot), \dots, \lambda_i(\cdot)\}^*, \qquad 
\xi_i(\cdot) := (d(x_0,\lambda_1(\cdot)),\dots,d(x_0,\lambda_i(\cdot)),0,0,\dots).
\]
Note that each $\xi_i : I \to \ell_\Phi(\R)$ is continuous by the reverse triangle inequality with respect to $d.$ Since $\Phi$ is regular, we have
$
\lim_{i \to \infty} \Phi(\xi_0(t) - \xi_i(t)) = 0
$
for each $t \in I.$ To prove the implication $\textnormal{(i)} \Rightarrow \textnormal{(ii)},$ suppose that the following convergence is uniform;
\[
\lim_{i \to \infty} \Phi(d(x_0,\lambda_{i+1}(\cdot)), d(x_0,\lambda_{i+2}(\cdot), \dots)
=
\lim_{i \to \infty} \Phi(\xi_0(\cdot) - \xi_i(\cdot))
= 0.
\]
The continuity of $\xi_0$ follows from that of each $\xi_i$ by the uniform limit theorem. To prove the implication $\textnormal{(ii)} \Rightarrow \textnormal{(iii)},$ we assume that $\xi_0$ is continuous. Observe that for each $i_0 = 0,1,2,\dots,$ the following ``cut-off mapping'' is obviously 1-Lipschitz continuous;
\[
\ell_\Phi(\R) \ni (\xi_1,\xi_2,\dots) \longmapsto (\xi_{i_0 +1},\xi_{i_0 + 2},\dots) \in \ell_\Phi(\R).
\]
It follows that for each $i_0$ the sequence $(d(x_0,\lambda_{i_0+i}(\cdot)))_{i \in \N}$ is continuous, from which the continuity of $d_\Phi(O_{x_0},(S - S_{i_0})(\cdot)) = \Phi[(d(x_0,\lambda_{i_0+i}(\cdot)))_{i \in \N}]$ follows. To prove the continuity of $S,$ we let $\epsilon > 0$ and $t_0 \in I$ be arbitrary. Since $\Phi$ is regular and $(d(x_0,\lambda_i(t_0)))_{i \in \N} \in \ell_\Phi(\R),$ there exists a large enough integer $i_0$ (depending on both $\epsilon$ and $t_0$) such that
\begin{equation}
\label{Equation0: Induce Criterion}
d_\Phi(O_{x_0}, (S - S_{i_0})(t_0))
    = \Phi[(d(x_0,\lambda_{i_0+i}(t_0)))_{i \in \N}] < \frac{\epsilon}{4}.
\end{equation}
Since $d_\Phi(O_{x_0}, (S -S_{i_0})(\cdot))$ is continuous at $t_0,$ there exists a neighborhood $I_0$ of $t_0$ with
\begin{equation}
\label{Equation1: Induce Criterion}
d_\Phi(O_{x_0}, (S - S_{i_0})(t)) < \frac{\epsilon}{2}, 
\qquad t \in I_0.
\end{equation}
Since $S_{i_0}$ is continuous at $t_0,$ we may shrink $I_0$ if necessary, to
ensure that
\begin{equation}
\label{Equation 2: Induce Criterion}
d_\Phi(S_{i_0}(t_0), S_{i_0}(t)) < \frac{\epsilon}{4}, 
\qquad  t \in I_0.
\end{equation}
It follows from \crefrange{Equation0: Induce Criterion}{Equation 2: Induce Criterion} that for all $t \in I_0$ we have
\begin{align*}
d_\Phi(S(t_0),S(t))
    &= d_\Phi((S - S_{i_0})(t_0) + S_{i_0}(t_0), (S - S_{i_0})(t)+ S_{i_0}(t)) \\
    &\leq d_\Phi((S - S_{i_0})(t_0),(S - S_{i_0})(t)) + d_\Phi(S_{i_0}(t_0),S_{i_0}(t)) \\
    &\leq d_\Phi((S - S_{i_0})(t_0), O_{x_0}) + d_\Phi(O_{x_0},(S - S_{i_0})(t)) + d_\Phi(S_{i_0}(t_0),S_{i_0}(t)) \\
    &< \frac{\epsilon}{4} + \frac{\epsilon}{2} + \frac{\epsilon}{4} = \epsilon,
\end{align*}
thereby establishing the continuity of $S$ at $t_0.$ To prove the last implication $\textnormal{(iii)} \Rightarrow \textnormal{(i)},$  we assume that $S$ is continuous. Then $S - S_i$ is continuous for each $i \in \N.$ It follows from the continuity of $d_\Phi$ that each $f_i := d_\Phi(O_{x_0},(S-S_i)(\cdot)) : I \to \R$ is continuous. By construction, $(f_i)_{i \in \N}$ is a pointwise non-increasing sequence of continuous functions, which converges to $0$ pointwise. Since $I$ is compact, it follows from Dini's theorem (see \cite[Theorem 7.13]{Book:Rudin76} for details) that $f_i \to 0$ uniformly.
\end{proof}

We have the following corollary of \cref{Theorem: Induce Criterion};

\begin{cor}
\label{Corollary: Subsequence Corollary}
Let $\Phi$ be regular, and let $I$ be compact. Let $(\lambda_i)_{i \in \N}$ be a pointwise $\Phi$-summable sequence of continuous $X$-valued mappings on $I,$ such that $S := \{\lambda_1,\lambda_2,\dots\}^*$ is continuous. Then the following assertions hold true:
\begin{enumerate}[(i)]
\item The sequence $(\lambda_i(\cdot))_{i \in \N}$ converges uniformly to $x_0.$
\item If $(\lambda'_i(\cdot))_{i \in \N}$ is a subsequence of $(\lambda_i(\cdot))_{i \in \N},$ then the induced mapping $S' := \{\lambda'_1,\lambda'_2,\dots\}^*$ is continuous.
\end{enumerate}
\end{cor}
\begin{proof}
Since $S= \{\lambda_1,\lambda_2,\dots\}^*$ is continuous, it follows from \cref{Theorem: Induce Criterion} that the following convergence is uniform;
\begin{equation}
\label{Equation1: Uniform Convergence Property}
\Phi(d(x_0,\lambda_{i+1}(\cdot)), d(x_0,\lambda_{i+2}(\cdot)), \dots) \to 0.
\end{equation}
For each fixed $\epsilon > 0,$ there exists a large enough integer $i_0 \in \N,$ such that for all $t \in I$ and all $i > i_0$ we have
\begin{equation}
\epsilon
>\Phi(d(x_0, \lambda_{i+1}(t)), d(x_0, \lambda_{i+2}(t), \dots)
\geq d(x_0, \lambda_{i+1}(t)).
\end{equation}
The first assertion follows. On the other hand, any subsequence $(\lambda'_i(\cdot))_{i \in \N}$ of $(\lambda_i(\cdot))_{i \in \N}$ also has the uniform convergence property \cref{Equation1: Uniform Convergence Property}. The second assertion follows from \cref{Theorem: Induce Criterion}. 
\end{proof}

We conclude the current section with the following remark;
\begin{rem}
\label{Remark: Induce Remark}
Let $\Phi$ be regular, and let $I$ be compact. Given a continuous $X$-valued mapping $\lambda$ defined on $I,$ we let
$
R(\lambda) := \sup_{t \in I} d(x_0,\lambda(t)).
$
With this notation in mind, a sequence $(\lambda_i(\cdot))_{i \in \N}$ of continuous $X$-valued mappings on $I$ converges uniformly $x_0$ if and only if $R(\lambda_i) \to 0$ as $i \to \infty.$ That is, given a continuous mapping $S : I \to \eS_\Phi(X,x_0)$ admitting a representation $S = \{\lambda_1, \lambda_2, \dots\}^*,$ where each $\lambda_j : I \to X$ is continuous, it follows from \cref{Corollary: Subsequence Corollary} (i) that no matter how small $\epsilon > 0$ we may be given, all but finitely many mappings in $\lambda_1, \lambda_2,\dots$ have their images entirely included in the open $\epsilon$-neighborhood of $x_0.$
\end{rem}

\section{Continuous Enumeration of Multiset-valued Mappings}
\label{Section: Continuous Enumeration of Multiset-valued Mappings}


\begin{notation}
Let $\Phi$ be a (not necessarily regular) symmetric norm, and let $(X,x_0)$ be a based metric space endowed with a metric $d.$ The function $\tan^{-1}$ gives an isometric isomorphism from the metric space $[-\infty,+\infty]$ onto $[-\pi/2, \pi/2],$ which allows us to identify any subinterval of $[-\infty,+\infty]$ with an interval of finite length.
\end{notation}

Recall that \cref{Theorem: Induce Criterion} gives a criterion under which a given sequence of continuous $X$-valued mappings $\lambda_1, \lambda_2,\dots$ implies the continuity of the associated $\eS_\Phi(X,x_0)$-valued mapping $\{\lambda_1, \lambda_2,\dots\}^*.$ A different kind of continuity result is the following;

\begin{thm}[existence of continuous enumeration]
\label{Theorem: Continuous Enumeration}
Let $I$ be any subinterval of $[-\infty,\infty],$ and let $S$ be a continuous $\eS_\Phi(X,x_0)$-valued mapping on $I.$ Then there exist infinitely many continuous $X$-valued mappings $\lambda_1, \lambda_2,\dots : I \to X,$ such that $S(t) = \{\lambda_1(t),\lambda_2(t),\dots\}^*$ for all $t \in [0,1].$
\end{thm}

Any sequence $(\lambda_i)_{i \in \N}$ of continuous $X$-valued mappings satisfying the conclusion of \cref{Theorem: Continuous Enumeration} will be called a \textbi{continuous enumeration} of $S.$ The purpose of the current section is to prove \cref{Theorem: Continuous Enumeration}.

\subsection{A strategy of the proof}

In this section we shall prove the following two technical results:

\begin{thm}[finite-rank continuous enumeration]
\label{Theorem: Finite-rank Continuous Enumeration}
Let $I$ be any subinterval of $[-\infty,\infty],$ and let $S$ be a continuous $\eS_\Phi(X,x_0)$-valued mapping on $I.$ If $S$ has a finite rank $n,$ then there exist $n$ continuous mappings $\lambda_1, \dots, \lambda_n : I \to X,$ such that $S = \{\lambda_1, \dots, \lambda_n\}^*.$
\end{thm}

\begin{pro}[$\epsilon$-separation]
\label{Theorem: Epsilon Separation}
Let $I$ be any compact subinterval of $[-\infty,\infty],$ and let $S$ be a continuous $\eS_\Phi(X,x_0)$-valued mapping on $I.$ Then for each $\epsilon>0,$ there exists a finite-rank continuous mapping $S_\epsilon : I \to \eS_\Phi(X,x_0),$ such that for all $t \in I$ we have $S_\epsilon(t) \subseteq S(t)$ and $S(t) -S_\epsilon(t) \subseteq B_{\epsilon}(x_0).$ Here, $B_{\epsilon}(x_0)$ denotes the open $\epsilon$-neighborhood around $x_0.$
\end{pro}

\cref{Theorem: Finite-rank Continuous Enumeration} is a multiset analogue of \cref{Theorem: Kato's Selection Theorem}, Kato's finite-dimensional continuous enumeration, while \cref{Theorem: Epsilon Separation} is inspired by \cref{Remark: Induce Remark}.  Note that \cref{Theorem: Epsilon Separation} does not hold true in general for non-compact intervals $I$ (see \cref{Example: Epsilon Separation Does Not Hold for Compact Intervals}).

\begin{proof}[Proof of \cref{Theorem: Continuous Enumeration} via \cref{Theorem: Finite-rank Continuous Enumeration} and \cref{Theorem: Epsilon Separation}]
(A) Let us assume first that $I$ is compact. We set $\epsilon_n := 1 / n$ for each $n \in \N$ and proceed inductively. It follows from \cref{Theorem: Epsilon Separation} that there exists a finite-rank continuous mapping $S_{\epsilon_1} : I \to \eS_\Phi(X,x_0),$ such that for all $t \in I$ we have $S_{\epsilon_1}(t) \subseteq S(t)$ and $S(t) -S_{\epsilon_1}(t) \subseteq B_{\epsilon_1}(x_0).$ We can then apply the same proposition to $S - S_{\epsilon_1}$ with $\epsilon = \epsilon_2,$ where the continuity of this mapping is guaranteed by \cref{Proposition: Continuity of Sum and Difference} (ii). In this way we obtain another finite-rank continuous mapping $S_{\epsilon_2} : I \to \eS_\Phi(X,x_0)$ satisfying the desired property. Proceeding this way, we can form a sequence of continuous mappings $S_{\epsilon_1},S_{\epsilon_2},\dots : I \to \eS_\Phi(X,x_0),$ where each $S_{\epsilon_i}$ admits a representation $S_{\epsilon_i} = \{\lambda^i_1,\dots,\lambda^i_{n_i}\}^*$ according to \cref{Theorem: Finite-rank Continuous Enumeration}. By construction, $S = \{\lambda^1_1,\dots,\lambda^1_{n_1},\lambda^2_1,\dots,\lambda^2_{n_2},\dots\}^*,$ and so the claim follows.

(B) Let us drop the assumption of $I$ being compact. Suppose that $S$ is a continuous $\eS_\Phi(X,x_0)$-valued mapping on a non-compact interval $I.$ That is, $I$ is either half-open or open. We can then construct an increasing sequence $(I_i)_{i \in \N}$ of compact intervals with $I = \bigcup_i I_i,$ where the restriction $S|_{I_i}$ admits a continuous enumeration $(\lambda^i_j)_{j \in \N}$ by Part (A). We may assume without loss of generality that for each $i \in \N$ we have $\lambda^{i}_j(t) = \lambda^{i+1}_j(t)$ for each $t \in I_i$ and each $j$ by using the obvious inductive argument. Let
\[
\lambda_j(t) = \lim_{i \to \infty} \lambda^i_j(t), \qquad t \in I.
\]
By construction, $(\lambda_j)_{j \in \N}$ is a continuous enumeration of $S.$ The proof is complete.
\end{proof}

The rest of the current section is devoted to proving \cref{Theorem: Finite-rank Continuous Enumeration} and \cref{Theorem: Epsilon Separation}.

\subsection{Existence of finite-rank continuous enumeration (Proof of \texorpdfstring{\cref{Theorem: Finite-rank Continuous Enumeration}}{})}

We shall first prove \cref{Theorem: Finite-rank Continuous Enumeration} by mimicking Kato's original proof for \cref{Theorem: Kato's Selection Theorem} in the language of multisets. Given a continuous $\eS_\Phi(X,x_0)$-valued mapping $S,$ there is no natural way to enumerate $S$ in general. However, there are some obvious cases;

\begin{lem}
\label{Theorem: Simple Continuous Enumeration}
Let $I$ be a metric space, and let $S : I \to X$ be a continuous mapping of finite rank $n.$ Suppose that for each $t \in I$ the multiset $S(t)$ admits a representation $S(t) = \{\lambda(t), \dots, \lambda(t)\}^*$ for some $\lambda(t) \in X.$ In this case, $\lambda(t)$ depends continuously on $t \in I.$ That is, $S$ admits a continuous enumeration.
\end{lem}
\begin{proof}
This follows from Inequality \cref{Equation: Phi-estimate}. 
\end{proof}

\begin{proof}[Proof of \cref{Theorem: Finite-rank Continuous Enumeration}]
We may assume without loss of generality that the domain $I$ has a finite length. For brevity, let us call the finite sequence $(\lambda_1,\dots,\lambda_n)$ satisfying the conclusion of \cref{Theorem: Finite-rank Continuous Enumeration} a \textbi{finite-rank continuous enumeration} of $S.$

(A) We prove one preliminary result beforehand. Let $I_1,I_2$ be two subintervals of $I,$ such that $I_1 \cup I_2$ is again a subinterval of $I.$ It is easy to see that if the restrictions $S|_{I_1},S|_{I_2}$ have finite-rank continuous enumerations $(\lambda^{1}_i)_{i=1}^n,(\lambda^{2}_i)_{i=1}^n$ respectively, then $S$ has a continuous enumeration on $I_1 \cup I_2.$ This is because the intersection $I_1 \cap I_2$ contains at least one point, say $t_0,$ and the two sequences $(\lambda^{1}_i(t_0))_{i=1}^n,(\lambda^{2}_i(t_0))_{i=1}^n$ are identical up to a permutation. It follows from this result that if $J$ is a subinterval of $I$ such that each point of $J$ has a neighborhood on which $S$ has a finite-rank continuous enumeration, then $S$ has a finite-rank continuous enumeration on the whole interval $J.$ Indeed, such $J$ has the property that the mapping $S$ admits a finite-rank continuous enumeration on any compact subinterval of $J,$ and so as in Part (B) of the proof of \cref{Theorem: Continuous Enumeration}, we can express $J$ as the union of increasing compact intervals.

(B) We proceed by induction on $n \geq 1.$ The base step $n = 1$ is done in \cref{Theorem: Simple Continuous Enumeration}. Suppose that the claim is proved for $n$ replaced by a smaller number and for any interval $I.$ Let $\Gamma$ be the set of all $t \in I$ such that $S(t)$ admits a representation $S(t) =
\{\lambda(t),\dots,\lambda(t)\}^*,$ where the point $\lambda(t) \in X$ is repeated $n$ times. It follows from \cref{Theorem: Simple Continuous Enumeration} again that $\lambda(t)$ depends continuously on $t \in \Gamma.$ Since $\Gamma$ is a closed subset of $I$ by \cref{Proposition: Reducing Open Sets}, the open set $I \setminus \Gamma$ can be written as countable union of pairwise disjoint open subintervals $I_1,I_2,\dots$ of $I.$ Given any such interval $I_j$ and any point $t_j \in I_j,$  since $\supp S(t_j)$
can be written as a union of two non-empty finite subsets of $X,$ it follows from the induction hypothesis and \cref{Proposition: Reducing Open Sets} that $t_j$ has a neighborhood on which a finite-rank continuous enumeration exists. It follows from (A) that a finite-rank continuous enumeration $(\lambda^{j}_1,\dots,\lambda^{j}_n)$ exists on each $I_j.$ We define  $\lambda_1,\dots,\lambda_n : I \to X$ by
\begin{equation}
\label{Definition of Lambda i}
\lambda_i(t):=
    \begin{cases}
        \lambda(t), & \mbox{if $t \in \Gamma,$}\\
        \lambda_i^{j}(t), & \mbox{if $t \in I_j,$ $j=1,2,\dots.$}
    \end{cases}
\end{equation}
We have $S = \{\lambda_1, \dots, \lambda_n\}^*$ by construction. It remains to prove the continuity of each $\lambda_i$ at $t_0 \in \Gamma.$ For such $t_0,$ it follows from \cref{Equation: Phi-estimate} that
\[
\sup_{1 \leq i \leq n} d(\lambda_i(t_0),\lambda_i(t)) = d(\lambda(t_0), \lambda_i(t)) \leq  2 \, d_\Phi(S(t_0), S(t)).
\]
Therefore, the continuity of $\lambda_i$ at $t_0 \in \Gamma$ follows from that of $S.$ The claim follows.
\end{proof}

\subsection{Proposition of \texorpdfstring{$\epsilon$}{epsilon}-separation (Proof of \texorpdfstring{\cref{Theorem: Epsilon Separation}}{})}

Let $S : I \to \eS_\Phi(X,x_0)$ be a continuous mapping, and let $\epsilon>0$ be fixed. A pair $(J,S_\epsilon)$ of a subinterval $J$ of $I$ and a finite-rank continuous mapping $S_\epsilon : J \to \eS_\Phi(X,x_0)$ is said to have \textbi{$\epsilon$-separation}, if for all $t \in J$ we have $S_\epsilon(t) \subseteq S(t)$ and $S(t) - S_\epsilon(t) \subseteq B_\epsilon(x_0).$ For brevity we also say that the interval $J$ has $\epsilon$-separation in this case. Let us first prove the following local version of \cref{Theorem: Epsilon Separation};

\begin{lem}
\label{Lemma: Local LFS}
Let $I$ be any subinterval of $[-\infty,\infty].$ If $S : I \to \eS_\Phi(X,x_0)$ is continuous and if $\epsilon>0$ is fixed, then any point $t_0 \in I$ has a neighborhood with $\epsilon$-separation.
\end{lem}
\begin{proof}
Given any $t_0 \in I$ there exists $\epsilon_0 \in (0,\epsilon),$ such that $S(t_0) \cap B_{\epsilon_0}(x_0) = S(t_0) \cap B_{\epsilon}(x_0)$ and the distance between $x_0$ and any point in $S(t_0)$ is never $\epsilon_0.$ If we set $U_0 := B_{\epsilon_0}(x_0),$ then there exists an open set $U_1$ in $X,$ such that $\dist(U_0,U_1) > 0$ and $(S(t_0) \setminus U_0) \subseteq U_1.$ Since $S(t_0) \subseteq U_0 \cup U_1,$ where $(U_0,U_1)$ is a positively separated pair of open sets, it follows that $t_0$ has a neighborhood $I_0$ with the properties (i),(ii),(iii) in \cref{Proposition: Reducing Open Sets}. Then $I_0$ has $\epsilon$-separation with respect to the restriction of $S \cap U_1$ to $I_0.$
\end{proof}

This allows us to prove the following selection lemma by using Zorn's lemma;
\begin{lem}
\label{Lemma: Extension Lemma}
Let $I = [a,b]$ be a compact interval in $[-\infty,\infty],$ and let $S$ be a continuous $\eS_\Phi(X,x_0)$-valued mapping on $I.$ If $\lambda_a \in S(a)$ (resp. $\lambda_b \in S(b)$), then there exists a continuous mapping $\lambda : I \to X$ with the property that $\lambda(t) \in S(t)$ for all $t \in I$ and $\lambda(a) = \lambda_{a}$ (resp. $\lambda(b) = \lambda_{b}$).
\end{lem}
\begin{proof}
We shall only consider the case $\lambda_a \in S(a),$ since the other result follows by symmetry. Let $\cC$ be the set of all those pairs of the form $(J, \lambda),$ where $J$ is a subinterval of $I$ containing $a$ and $\lambda$ is a continuous $X$-valued mapping on $J,$ such that $\lambda(t) \in S(t)$ for all $t \in J$ and $\lambda(a) = \lambda_{a}.$ We define the partial order $\preceq$ on $\cC$ as follows;
\begin{equation}
\label{Equation: Definition of Partial Order}
(J,\lambda) \preceq (J', \lambda') \mbox{ if and only if } J \subseteq J' \mbox{ and } \lambda' \mbox{ restricted to $J$ is $\lambda.$}
\end{equation}
By Zorn's lemma, the set $\cC$ contains a maximal element $(J,\lambda),$ where $J$ is either of the following forms: $J = [a,t_0]$ or $J = [a,t_0).$ The first case $J = [a, t_0]$ is impossible unless $t_0 = b,$ since \cref{Lemma: Local LFS} and \cref{Theorem: Finite-rank Continuous Enumeration} would allow us to continuously extend the domain of $\lambda.$ 

It remains to show that the second case $J = [a,t_0)$ is contradictory by proving that $\lambda$ can be continuously prolonged to $[a, t_0].$ Since $J$ is maximal, we have that $\lambda(t)$ does not tend to $x_0$ as $t \to t_0.$ Then there exists a neighborhood $U_0$ of $x_0,$ such that any neighborhood of $t_0$ in $[a,t_0]$ contains some point $t'_0$ with $\lambda(t'_0) \notin U_0.$ Since $\eS^{U_0}_\Phi(X,x_0)$ is an open subset of $\eS_\Phi(X,x_0)$ by \cref{Lemma: Openness of Restricted Sp}, it follows that $S(t_0) \setminus U_0$ must contain at least one point other than the basepoint $x_0.$ This allows us to write
\[
\supp (S(t_0) \setminus U_0) = \{x_0, x_1, \dots, x_k\}, \qquad x_i \neq x_j.
\]
We can then choose neighborhoods $U_1, \dots, U_k$ of $x_1, \dots, x_k$ respectively, in such a way that $(U_0, \dots, U_k)$ forms a positively separated tuple of open sets. It follows that there exists an open subinterval $I_0$ of $I,$ such that $t_0 \in I_0$ and the properties (i),(ii),(iii) in \cref{Proposition: Reducing Open Sets} hold true. Let us consider the following disjoint sets:
\[
\Delta_i := \{t \in I_0 \cap J \mid \lambda(t) \in U_i\}, \qquad i = 0, \dots, k.
\]
Note that each $\Delta_i$ is an open subset of $I_0 \cap J.$ Since $I_0 \cap J = \bigcup_{i=0}^k \Delta_i,$ each $\Delta_i$ is also a closed subset of the connected space $I_0 \cap J.$ It follows that there exists $i_0 = 1, \dots, k,$ such that $I_0 \cap J = \Delta_{i_0}$ and $\Delta_{i} = \emptyset$ for each $i \neq i_0.$ The condition (ii) in \cref{Proposition: Reducing Open Sets} gives for each $t \in I_0 \cap J$ we have
\[
d_\Phi(S(t),S(t_0)) \geq d_\Phi(S(t) \cap U_{i_0}, S(t_0) \cap U_{i_0}) = d_\Phi(S(t) \cap U_{i_0}, \{x_{i_0}, \dots, x_{i_0}\}^*) \geq \frac{1}{2} d(\lambda(t), x_{i_0}),
\]
where the last inequality follows from \cref{Equation: Phi-estimate}. Since $t_0$ is an accumulation point of the interval $I_0 \cap J$ by construction, setting $\lambda(t_0) := x_{i_0}$ continuously prolongs the domain of $\lambda$ to $[a_0, t_0].$ This contradicts the maximality of $\lambda,$ and so the second case $J = [a,t_0)$ is impossible. It follows that $J = I.$ 
\end{proof}

We are now in a position to prove \cref{Theorem: Epsilon Separation};

\begin{proof}[Proof of \cref{Theorem: Epsilon Separation}]
(A) We may assume without loss of generality that $I = [0,1].$ It follows from \cref{Lemma: Local LFS} that each $t \in [0,1]$ has a neighborhood $I_{t}$ with $\epsilon$-separation. Suppose that the open cover $\{I_t\}_{t \in [0,1]}$ has Lebesgue number $L > 0.$ We can then choose large enough $N \in \N,$ so that $1/N < L.$ It follows that each of the following intervals is contained in one of the members of the cover $\{I_t\}_{t \in [0,1]}:$ 
\[
\left[0,\frac{1}{N}\right], \qquad 
\left[\frac{1}{N},\frac{2}{N}\right], \qquad \dots, \qquad  \left[\frac{N - 1}{N},1\right].
\]
That is, each of the above intervals has $\epsilon$-separation. It remains to prove the following assertion.

(B) We show that if two subintervals $I_1,I_2$ of $[0,1]$ with $I_1 \cap I_2 = \{t_0\}$ have $\epsilon$-separation, then so does their union $I_1 \cup I_2.$ Let us assume that $I_1$ is located to the left of $I_2,$ and that $(I_1,S_1),(I_2,S_2)$ have $\epsilon$-separation. Note that $S_1,S_2$ admit finite-rank continuous enumerations 
$
S_1 = \{\lambda_1^{1},\dots,\lambda^{1}_{n_1}\}^*$ and
$S_2 = \{\lambda_1^{2},\dots,\lambda^{2}_{n_2}\}^*
$
according to \cref{Theorem: Finite-rank Continuous Enumeration}. We will construct a new finite-rank continuous mapping $S_\epsilon$ out of $S_1, S_2,$ so that $(I_1 \cup I_2, S_\epsilon)$ has $\epsilon$-separation. This process consists of the following two major steps. Firstly, we may assume that after a suitable rearrangement of the second sequence $(\lambda_i^{2}(t_0))_{i=1}^{n_2},$ we get
\begin{equation}
\label{Equation: LFS First Equation}
(\lambda_1^{1}(t_0),\dots,\lambda_n^{1}(t_0)) = (\lambda_1^{2}(t_0),\dots,\lambda_n^{2}(t_0)),
\end{equation}
where $n$ is the largest natural number such that \cref{Equation: LFS First Equation} holds true. We define the continuous mappings
$\lambda_1,\dots,\lambda_n:I_1 \cup I_2 \to X$ by
\[
\lambda_i(t) =
    \begin{cases}
        \lambda^{1}_i(t), & \mbox{ if $t \leq t_0,$}\\
        \lambda^{2}_i(t), & \mbox{ if $t > t_0.$}
    \end{cases}
\]
Secondly, it follows from \cref{Equation: LFS First Equation} that
$
\{\lambda^{1}_{n+1}(t_{0}),\dots,\lambda^{1}_{n_1}(t_{0})\}^* \subseteq S(t_0) - S_2(t_0).
$
Note that \cref{Lemma: Extension Lemma} allows us to continuously extend the domain of each $\lambda^{1}_{n+1}, \dots, \lambda^{1}_{n_1}$ from $I_1$ to $I_1 \cup I_2,$ in such a way that $\{\lambda^{1}_{n+1}(t), \dots, \lambda^{1}_{n_1}(t)\}^* \subseteq S(t) - S_2(t)$ for each $t \in I_2.$ Similarly, we can continuously extend the domain of each $\lambda^{2}_{n+1}, \dots, \lambda^{2}_{n_2}$ from $I_2$ to $I_1 \cup I_2,$ in such a way that $\{\lambda^{2}_{n+1}(t), \dots, \lambda^{2}_{n_2}(t)\}^* \subseteq S(t) - S_1(t)$ for each $t \in I_1.$ It is now easy to see that $I_1 \cup I_2$ has $\epsilon$-separation with the following finite-rank continuous mapping; 
\[
S_\epsilon := \{\lambda_1,\dots,\lambda_n\}^* + \{\lambda^{1}_{n+1}, \dots, \lambda^{1}_{n_1}\}^* + \{\lambda^{2}_{n+1}, \dots, \lambda^{2}_{n_2}\}^*.
\]
\end{proof}

\section{Spectral Flow for Invertible Operators}
\label{Section: Flow of Discrete Spectrum}

We give an intuitive exposition of spectral flow for invertible operators as an application of \cref{Theorem: Continuous Enumeration} in this section. This will be done in the following major steps. For simplicity, we restrict our attention to a certain class of invertible operators which are compact perturbations of the identity operator $1.$ Given a continuous one-parameter family $\{T(t)\}_{t \in [0,1]}$ in such a fixed class, we use known results on the continuity of spectra to conclude that the associated family $\{\sigma(T(t))\}_{t \in [0,1]}$ is continuous with respect to $\eS_\infty(\C^*, 1),$ where $\C^* = \C\setminus\{0\}.$  Next, we introduce a homotopy invariant in $\eS_\infty(\C^*, 1)$ as follows. For a continuous family $S = \{S(t)\}_{t\in[0,1]}$ in $\eS_\infty(\C^*, 1)$ admitting a continuous enumeration $S = \{\lambda_1, \lambda_2, \dots\}^*$ and for each $\theta \in (0,2\pi),$ we define $\mu\left(\theta; \{S(t)\}_{t\in[0,1]} \right) \in \Z$ as the net number of times that the paths $\lambda_j$ cross the given ray $\R_+e^{i\theta}$ in the anticlockwise direction. We can then define the spectral flow of $\{T(t)\}_{t \in [0,1]}$ by
\begin{equation}
\label{Equation: Spectral Flow of T}
\flow(\theta; \{T(t)\}_{t \in [0,1]}) := \mu(\theta; \{\sigma(T(t))\}_{t \in [0,1]}), \qquad \theta \in (0,2\pi),
\end{equation}
This will allow us to recover the previously mentioned formula \cref{Equation: Naive Definition of Unitary Spectral Flow} as a special case where each $T(t)$ is unitary. In particular, if $\{T(t)\}_{t\in[0,1]}$ is a loop of the form $T(0) = T(1) = 1,$ then its spectral flow reduces to the net winding number of its eigenvalues as in \cref{Equation: Naive Definition of Spectral Flow for Loops}.

\subsection{Preliminaries} 
\label{Section: Preliminaries of Flow of Discrete Spectrum}

\begin{thm}[{\cite[Theorem VI.4.1]{Book:Bhatia97}}]
\label{Theorem: Hoffman-Wielandt Inequality}
If $N, N'$ are two $n \times n$ normal matrices, then their eigenvalues can be enumerated as $(\lambda_1,\dots,\lambda_n),\ (\lambda'_1,\dots,\lambda'_n)$ respectively, so that
\begin{equation}
\label{Equation: Hoffman-Wielandt Inequality}
\left(\sum^n_{i=1} |\lambda_i - \lambda'_i|^2\right)^{1/2}
  \leq \|N - N'\|_{\Phi_2},
\end{equation}
where $\|A\|_{\Phi_2} := \sqrt{\tr(A^*A)}$ is the Hilbert-Schmidt norm.
\end{thm}

Inequality \cref{Equation: Hoffman-Wielandt Inequality} is known as the \textbi{Hoffman-Wielandt inequality} for normal matrices, and we shall consider a certain infinite-dimensional analogue of this result (\cref{Theorem: Bhatia-Sinha}). In order to state this theorem we recall some preliminary results first.

Let $\cH$ be a fixed separable Hilbert space. Following \cite[\textsection XIII.4]{Book:ReedSimon4}, we define the \textbi{essential spectrum} of a bounded operator $T$ on $\cH,$ denoted by $\ess(T),$ as the complement of the \textbi{discrete spectrum} $\dis(T).$ That is, $\ess(T)= \sigma(T) \setminus \dis(T),$ where $\dis(T)$ is the set of all those complex numbers $z \in \C,$ such that $z$ is a discrete point of the spectrum $\sigma(T)$ and $z$ is an eigenvalue of $T$ with finite algebraic multiplicity. By Weyl's theorem (see, for example, \cite[Theorem XIII.14]{Book:ReedSimon4}) the essential spectrum of a self-adjoint operator is invariant under compact perturbations. In particular, for each compact operator $A$ on $\cH$ we have 
$
\ess(1 + A) = \ess(1) = \{1\}.
$

Let $\Phi$ be a symmetric norm. The \textbi{singular numbers} of a compact operator $A$ on $\cH,$ denoted by $s_1(A),s_2(A),\dots,$ are the eigenvalues of the positive operator $|A| := \sqrt{A^*A}$ repeated according to their multiplicities and arranged in non-increasing order. 
 The operator $A$ is said to be \textbi{$\Phi$-summable}, if $(s_i(A))_{i \in \N} \in \ell_\Phi(\R).$ That is,
\begin{equation}\label{Phi-Norm}
\|A\|_{\Phi} := \lim_{n \to \infty} \Phi(s_1(A),\dots, s_n(A),0,0,\dots) < \infty.
\end{equation}
The set $\fS_\Phi(\cH)$ of all $\Phi$-summable operators, known
as the \textbi{$\Phi$-Schatten class}, forms a Banach space with the
norm \cref{Phi-Norm} (see, for example, \cite[Theorem III.4.1]{Book:Gohberg-Krein69}). In particular, the \textbi{$p$-Schatten class} is the Banach space $\fS_p(\cH) := \fS_{\Phi_p}(\cH),$ where $\Phi_p$ is given by \cref{Equation: Definition of p-norm}.

The problem of extending the Hoffman-Wielandt inequality \cref{Equation: Hoffman-Wielandt Inequality} to infinite dimensions appears in \cite{Kato87}, where an \textbi{extended enumeration} of the discrete spectrum of a bounded operator $T$ is defined as any infinite sequence $(\lambda_i)_{i \in \N}$ whose terms consist of all members of $\dis(T),$ each of which is repeated according to its algebraic multiplicity, and in addition, the sequence may contain some points of the essential spectrum $\ess(T).$

\begin{thm}[{\cite{Bhatia-Sinha88}}]
\label{Theorem: Bhatia-Sinha}
Let $\cH$ be a separable Hilbert space, and let $\Phi$ be a symmetric norm. For any pair $U,U'$ of unitary operators on $\cH$
with $U - U' \in \fS_\Phi(\cH),$ there exists a pair $(\lambda_i)_{i \in \N},(\lambda'_i)_{i \in \N}$ of extended enumerations of $\dis(U),\dis(U')$ respectively, such that
\[
\Phi(|\lambda_1 - \lambda'_1|,|\lambda_2 - \lambda'_2|,\dots)
\leq \frac{\pi}{2}\|U - U'\|_\Phi.
\]
\end{thm}

Note that there are some other variants of this infinite-dimensional Hoffman-Wielandt inequality (cf. \cite[Theorem II]{Kato87} and \cite[Corollary 2.3]{Bhatia-Davis99}).

\subsection{Continuity of the spectrum}

\begin{notation}
Let $\Phi$ be a regular symmetric norm, and let $\cH$ be a separable infinite-dimensional Hilbert space. Let $\C^* = \C\setminus\{0\},$ and let $\T = \{z \in \C \mid |z| = 1\}.$ 
\end{notation}

Let us consider the group $\eG_\infty(\cH,1)$ of invertible operators which differ from the identity operator $1$ by a compact operator. That is, $T \in \eG_\infty(\cH,1)$ if and only if $0 \notin \sigma(T)$ and $T - 1 \in \fS_\infty(\cH).$ The group $\eG_\infty(\cH,1)$ inherits the uniform norm $\|\cdot\|_\infty$ from $\cB(\cH);$ 
\[
\dist(T,T') := \|T - T'\|_\infty, \qquad T,T' \in \eG_\infty(\cH,1).
\]
Recall that any $T \in \eG_\infty(\cH,1)$ shares the same essential spectrum $\ess(T) = \{1\},$ and so the spectrum of $T$ can be identified with the following multisubset of $(\C^*,1);$
\begin{equation}
\label{Equation: Identify Spectrum with Multiset}
\sigma(T) = \{\lambda_1, \lambda_2, \dots\}^*,
\end{equation}
where $(\lambda_i)_{i \in \N}$ is any extended enumeration of $T.$ Alternatively, we may restrict our attention to the subgroup $\eU_\Phi(\cH,1)$ consisting of all those unitary operators $U$ on $\cH$ with $U - 1 \in \fS_\Phi(\cH).$ In this case, a stronger complete metric is given by the $\Phi$-norm;
\[
\dist_\Phi(U,U') := \|U - U'\|_\Phi \geq \|U -  U'\|_\infty, \qquad U,U' \in \eU_\Phi(\cH,1).
\]

\begin{pro}
\label{Theorem: Continuous Spectra}
With the identification \cref{Equation: Identify Spectrum with Multiset} in mind, the following map is continuous;
\[
\eG_\infty(\cH,1) \ni T \longmapsto \sigma(T) \in \eS_\infty(\C^*,1).
\] 
In addition, the map $\eU_\Phi(\cH,1) \ni U \longmapsto \sigma(U) \in \eS_\Phi(\T,1)$ is $\pi/2$-Lipschitz continuous;
\begin{equation}
\label{Equation: Lipschitz Bhatia-Sinha}
d_\Phi(\sigma(U),\sigma(U')) \leq \frac{\pi}{2} \|U - U'\|_{\Phi},
\qquad U,U' \in \eU_\Phi(\cH,1).
\end{equation}
\end{pro}
\begin{proof}
For the first assertion, recall that the non-zero eigenvalues of compact operators are uniformly continuous in the following sense (see, for example, \cite[Lemma XI.9.5]{Book:Dunford-Schwartz88}). If $A_0,A_1,A_2, \dots$ are compact and $\|A_n - A_0\|_\infty \to 0$ as $n\to\infty,$ then there exist enumerations of their non-zero eigenvalues $(\lambda_j(A_n))_{j \in \N},$ such that 
\[
\sup_{j \in \N} |\lambda_j(A_n) - \lambda_j(A_0)| \to 0.
\]
As for the second assertion, the estimate \cref{Equation: Lipschitz Bhatia-Sinha} immediately follows from \cref{Theorem: Bhatia-Sinha}. 
\end{proof}

\cref{Theorem: Baby Continuous Enumeration} is part of the following corollary;
\begin{cor}
\label{Theorem: Big Continuous Enumeration}
Let $I$ be any subinterval of $[-\infty,\infty],$ and let $\Phi$ be a symmetric norm. Then the following assertions hold true:
\begin{enumerate}
\item If $\{T(t)\}_{t \in I}$ is a continuous one-parameter family  in $\eG_\infty(\cH,1),$ then $\{\sigma(T(t))\}_{t \in I}$ is continuous in $\eS_\infty(\C^*,1).$ Moreover, there exist infinitely many continuous functions $\lambda_1,\lambda_2, \dots : I \to \C^*$ forming a continuous enumeration of $\{\sigma(T(t))\}_{t \in I}.$
\item If $\{U(t)\}_{t \in I}$ is a continuous one-parameter family  in $\eU_\Phi(\cH,1),$ then $\{\sigma(U(t))\}_{t \in I}$ is  $\pi/2$-Lipschitz continuous in $\eS_\Phi(\T,1).$ Moreover, there exist infinitely many continuous functions $\lambda_1,\lambda_2, \dots : I \to \T$ forming a continuous enumeration of $\{\sigma(U(t))\}_{t \in I}.$
\end{enumerate}
\end{cor}
\begin{proof}
This follows from \cref{Theorem: Continuous Spectra} and \cref{Theorem: Continuous Enumeration}.
\end{proof}

\begin{rem}
\label{Remark: Identification}
We can view $\eS_\Phi(\T,1)$ as a subspace of $\eS_\infty(\C^*,1)$ in the following precise sense. The based inclusion map $\iota : (\T,1) \to (\C^*,1)$ induces $\iota_* : \eS_\Phi(\T,1) \to \eS_\Phi(\C^*,1)$ as in \cref{Section: Multiset functors}, where $\eS_\Phi(\C^*,1) \subseteq \eS_\infty(\C^*,1).$ The identification of $\eS_\Phi(\T,1)$ with $\iota_*(\eS_\Phi(\T,1))$ allows us to restrict our attention to $\eS_\infty(\C^*, 1)$ from this point onward.
\end{rem}

\subsection{\texorpdfstring{\(\mu\)}{Mu}-invariant}

\begin{notation}
By a (continuous) \textbi{path} in a topological space $X,$ we mean any continuous $X$-valued mapping on $[0,1].$ The notation $\divideontimes$ denotes concatenation of paths and $\gamma^{-1}$ denotes the reverse of a given path $\gamma.$ As for the basics of Algebraic Topology, such as the definition of fundamental group, the reader is referred to \cite[\textsection 1-2]{Book:Hatcher02}.
\end{notation}

For $\theta\in(0,2\pi)$ and $x,x' \in \R,$ we let
\[
[\theta;x,x'] := \#\{k \in \Z \mid x \leq \theta + 2\pi k < x'\} - \#\{k \in \Z \mid x' \leq \theta + 2\pi k < x\},
\]
where $\#$ denotes cardinality. In words, $[\theta;x,x']$ is the number of points of $\theta+2\pi\Z$ which lie between $x$ and $x',$ taken with a negative sign if $x' < x.$  Given a path $S$ in $\eS_\infty(\C^*,1)$ admitting a continuous enumeration $S = \{r_1e^{i x_1},r_2e^{i x_2},\dots\}^*$ according to \cref{Theorem: Continuous Enumeration}, where each $r_j e^{i x_j}$ is a path of complex numbers in polar form, we introduce 
\begin{equation}
\label{Definition: Mu-invariant}
\mu\left(\theta,\{S(t)\}_{t \in [0,1]} \right) := \sum^\infty_{j=1} [\theta; x_j(0), x_j(1)], \qquad \theta \in (0,2\pi).
\end{equation}
This formula is motivated by the naive definition of unitary spectral flow \cref{Equation: Naive Definition of Unitary Spectral Flow}. Well-definedness of \cref{Definition: Mu-invariant} is part of the following theorem;

\begin{thm}
\label{Theorem: Properties of Mu-invariant}
The formula \cref{Definition: Mu-invariant} is well-defined in the sense that it does not depend on the choice of a continuous enumeration of $S,$ and that $[\theta;x_j(0),x_j(1)] = 0$ for large enough $j.$ Furthermore, if $S,T$ are two paths in $\eS_\infty(\C^*,1),$ then:
\begin{enumerate}[(i)]
  \item If $S,T$ are path-homotopic, then $\mu\left(-,\{S(t)\}_{t \in [0,1]} \right) = \mu\left(-,\{T(t)\}_{t \in [0,1]} \right).$
  \item If $S(1) = T(0),$ then $\mu\left(-,\{(S \divideontimes T) (t)\}_{t \in [0,1]} \right)  = \mu\left(-,\{S(t)\}_{t \in [0,1]} \right) + \mu\left(-,\{T(t)\}_{t \in [0,1]} \right).$
  \item We have $\mu\left(-,\{S^{-1}(t)\}_{t \in [0,1]} \right) = - \mu\left(-,\{S(t)\}_{t \in [0,1]} \right).$
  \item If $S(0) = S(1) = O_1,$ then $\mu\left(-,\{S(t)\}_{t \in [0,1]} \right) = \mathrm{const}.$
\end{enumerate}
\end{thm}

This theorem will allow us to assign to each path $S = \{S(t)\}_{t \in [0,1]}$ in $\eS_\infty(\C^*,1)$ a homotopy invariant $\mu\left(-,\{S(t)\}_{t \in [0,1]} \right) : (0,2\pi) \to \Z$ known as the \textbf{$\mu$-invariant} of the path $S.$

\begin{lem}
\label{Lemma: Finite Sum}
If $S$ is a path in $\eS_\infty(\C^*,1)$ admitting a continuous enumeration $S = \{r_1 e^{i x_1}, r_2 e^{i x_2}, \dots\}^*$ and if $\theta \in (0,2\pi)$ is fixed, then there exists a large enough integer $j_0 \in \N$ such that for each $j \geq j_0$ we have $[\theta; x_j(0),  x_j(1)] = 0.$
\end{lem}
That is, the right hand side of the formal notation \cref{Definition: Mu-invariant} is a finite sum of integers.
\begin{proof}
There exists a small enough $\epsilon > 0,$ such that the open $\epsilon$-neighborhood of $1,$ denoted by $B_\epsilon(1),$ has the property that its closure does not intersect with the rays $\R_+ e^{\pm i \theta}.$ It follows from \cref{Remark: Induce Remark} that there exists  $j_0 \in \N$ such that for each $j \geq j_0$ and each $t \in [0,1]$ we have $r_j(t) e^{i x_j(t)} \in B_\epsilon(1).$ For such $j,$ we have $[\theta; x_j(0), x_j(1)] = 0$ by construction.
\end{proof}

Next, we consider the case where the given path $S = \{r_1 e^{i x_1}, r_2 e^{i x_2}, \dots\}^*$ is a loop based at $O_1 = \{1,1, \dots\}^*.$ That is, $S(0) = S(1) = O_1,$ and so each $r_j e^{i x_j}$ is a loop based at $1.$ In this case, it is not difficult to observe that each $[\theta;x_j(0),x_j(1)]$ represents the \textbi{winding number} of the loop $r_j e^{i x_j}$ about $0.$ That is, $[\theta;x_j(0),x_j(1)] = (x_j(1) - x_j(0))/2\pi$ for each $\theta \in (0,2\pi).$ Note that \cref{Definition: Mu-invariant} has been reduced to the following integer which does not depend on $\theta \in (0,2\pi);$
\begin{equation}
\label{Equation: Definition of mu-invariant for Loops}
\mu\left(\theta,\{S(t)\}_{t \in [0,1]} \right) = \sum^\infty_{j=1} [\theta; x_j(0), x_j(1)] =  \sum^\infty_{j=1} \frac{x_j(1) - x_j(0)}{2\pi} = : \mu\left(\{S(t)\}_{t \in [0,1]} \right).
\end{equation}
We shall make use of the following non-trivial result without proof;

\begin{thm}
\label{Theorem: Flow in Circle Case}
The formula \cref{Equation: Definition of mu-invariant for Loops}, which assigns to each loop $S = \{S(t)\}_{t \in [0,1]}$ in the based topological space $(\eS_\infty(\C^*,1),O_1)$ a unique integer $\mu\left(\{S(t)\}_{t \in [0,1]} \right),$ does not depend on the choice of a continuous enumeration of $S.$ In fact, it induces the following group isomorphism;
\begin{equation}
\label{Equation: Special Case of Dold Thom}
\pi_1(\eS_\infty(\C^*,1),O_1) \ni \left[\{S(t)\}_{t \in [0,1]} \right]_{\pi_1} \longmapsto  \mu\left(\{S(t)\}_{t \in [0,1]} \right) \in \Z,
\end{equation}
where $\pi_1(X,x_0)$ is the fundamental group of a based topological space $(X,x_0)$ and where $\left[\{\gamma(t)\}_{t \in [0,1]} \right]_{\pi_1}$ is the homotopy class represented by a loop $\gamma$ with $\gamma(0) = \gamma(1) = x_0.$
\end{thm}

\begin{rem}
\cref{Theorem: Flow in Circle Case} can be generalised to other multiset spaces $\eS_\Phi(X,x_0).$  Indeed, \cref{Equation: Definition of mu-invariant for Loops} motivates us to introduce an explicit group isomorphism $\pi_1(\eS_\Phi(X,x_0),O_{x_0}) \simeq H_1(X)$ via continuous enumeration, where $H_1(X)$ denotes the first singular homology group of $X.$ This concrete analogue of the \textit{Dold-Thom theorem} (see, for example, \cite[\textsection 4.K]{Book:Hatcher02}), $\pi_1(\SP^\infty(X,x_0)) \simeq H_1(X),$ is the main subject of another paper in preparation (rigorous proof can be found in Y.\,T.'s master's thesis \cite{Tanaka14}). We can then recover \cref{Theorem: Flow in Circle Case} as a special case $\pi_1(\eS_\infty(\C^*,1),O_{1}) \simeq H_1(\C^*) \simeq \Z,$ where the last group isomorphism is given by the winding number for loops in $\C^*.$
\end{rem}

It remains to show that $\mu\left(-,\{S(t)\}_{t \in [0,1]} \right)$ does not depend on the choice of a continuous enumeration of $S = \{S(t)\}_{t \in [0,1]}$ with the aid of \cref{Theorem: Flow in Circle Case}. For $\theta \in (0,2\pi)$ and for $x_0 \in \R,$ we define $\eta_\theta(-;x_0)$ to be the straight path in $\R$ from $x_0$ to the nearest integer multiple of $2 \pi$ which does not cross $\theta + 2\pi\Z.$ More precisely, writing $x_0 = \theta_0 + 2\pi k_0$ with $\theta_0 \in [0,2\pi)$ and $k_0 \in \Z,$ we let for each $t \in [0,1]$
\[
\eta_\theta(t;x_0) := 
\begin{cases}
x_0(1-t) + 2 \pi k_0 t, &\mbox{if } \theta_0 \leq \theta, \\
x_0(1-t) + 2\pi (k_0+1) t, &\mbox{if } \theta_0 > \theta.
\end{cases}
\]
Given $r_0 > 0 $ and $x_0 \in \R,$ we denote by $\gamma_\theta(-;r_0,x_0)$ the following path from $r_0 e^{i x_0}$ to $1;$ 
\[
\gamma_\theta(t;r_0,x_0) := (r_0(1-t) + t) \cdot \exp(i \eta_\theta(t;x_0)), \qquad t \in [0,1], \qquad \theta \in (0,2\pi).
\]

\begin{lem}
\label{Lemma: Mu-invariant for Paths}
If $S_0$ is a fixed multisubset in $\eS_\infty(\C^*,1)$ admitting an enumeration $S_0 = \{s_1, s_2,\dots\}^*,$  then the following path is a well-defined continuous path in $\eS_\infty(\C^*,1)$ from $S_0$ to $O_1$ for each $\theta \in (0,2\pi);$
\begin{equation}
\label{Equation: Path-connected}
\Gamma_{\theta}(S_0) := \left\{\gamma_\theta(-; |s_1|, \arg s_1), \gamma_\theta(-; |s_1|, \arg s_2), \dots \right\}^*.
\end{equation}
Furthermore, if $S$ is a path in $\eS_\infty(\C^*,1)$ and if $\theta \in (0,2\pi),$ then we have
\begin{align}
&T_\theta := \Gamma_{\theta}(S(0))^{-1} \divideontimes S \divideontimes \Gamma_{\theta}(S(1)), \\
\label{Equation: Definition of mu-invariant for Paths}
&\mu \left( \{T_\theta(t)\}_{t \in [0,1]} \right) =  \sum^\infty_{j=1} [\theta;x_j(0),x_j(1)],
\end{align}
where $(r_j e^{i x_j})_{j \in \N}$ is any continuous enumeration of the given path $S.$ 
\end{lem}
The path \cref{Equation: Path-connected} can be used to show that $\eS_\infty(\C^*,1)$ is path-connected. 
\begin{proof}
Let us first show that \cref{Equation: Path-connected} is a well-defined path. Let $S_0 = \{s_1, s_2, \dots\}^*$ be a fixed multiset in $\eS_\infty(\C^*,1),$ and let $B_\epsilon(1)$ be the open $\epsilon$-neighborhood of $1$ as in the proof of \cref{Lemma: Finite Sum}. Since $s_j \to 1$ as $j \to \infty,$ there exists $j_0 \in \N$ such that for each $j \geq j_0$ we have $s_j \in B_\epsilon(1).$   It is then geometrically obvious that 
$
|\gamma_\theta(t;|s_j|, \arg s_j) - 1| \leq 
|s_j - 1|
$
for each $t \in [0,1]$ and each $j \geq j_0.$ It follows that the multiset $\{\gamma_\theta(t; |s_1|, \arg s_1), \gamma_\theta(t;|s_2|, \arg s_2),\dots\}^*$ is $\Phi$-summable for each $t \in [0,1].$ Moreover, 
\[
\sup_{t \in [0,1]} \Phi\left(\left(\gamma_\theta(t;|s_{j_0+j}|, \arg s_{j_0+j}) - 1\right)_{j \in \N}\right) 
\leq \Phi\left(\left(s_{j_0 + j} - 1\right)_{j \in \N}\right) \to 0, \qquad j_0 \to \infty.
\]
This ensures the continuity of $\Gamma_\theta(S_0)$ by \cref{Theorem: Induce Criterion}. Let $S$ be a path in $\eS_\infty(\C^*,1),$ and let $(r_j  e^{i x_j })_{j \in \N}$ be any continuous enumeration of the given path $S.$ Observe that the loop $T_\theta = \Gamma_\theta(S(0))^{-1} \divideontimes S \divideontimes \Gamma_\theta(S(1))$ is continuously enumerated by 
\[
z_j := \gamma_\theta(-;r_j(0),x_j(0))^{-1} \divideontimes r_j e^{ix_j} \divideontimes \gamma_\theta(-;r_j(1),x_j(1)), \qquad j \in \N.
\]
Since $z_j = |z_j| \exp(i \eta_\theta(-;x_j(0))^{-1} \divideontimes x_j \divideontimes \eta_\theta(-;x_j(1)))$ for each $j,$ we obtain
\begin{equation}
\label{Equation: Mu-invariant of Ttheta}
\mu(\{T_\theta(t)\}_{t \in [0,1]}) = \sum^\infty_{j=1} \frac{z_j(1) - z_j(0)}{2\pi} = \sum^\infty_{j=1} \frac{\eta_\theta(1;x_j(1)) - \eta_\theta(1;x_j(0))}{2\pi} .
\end{equation}
It therefore remains to verify the following non-trivial equality;

\begin{equation}
\label{Equation1: Infinitesimal Spectral Flow}
\frac{\eta_\theta(1; x') - \eta_\theta(1; x)}{2\pi} = [\theta;x,x'], \qquad x, x' \in \R.
\end{equation}
Let $x = \theta_0 + 2\pi k_0$ and $x' = \theta_0' + 2\pi k_0'$ with $\theta_0,\theta_0' \in [0,2\pi)$ and $k_0,k_0' \in \Z.$  Clearly, from the definition of $\eta_\theta,$ we have the following four possibilities:
\begin{equation}
\label{Equation2:  Infinitesimal Spectral Flow}
\frac{\eta_\theta(1;x') - \eta_\theta(1;x)}{2\pi}
= 
\begin{cases}
k_0' - k_0, &\mbox{if } \theta_0,\theta_0' \leq \theta,\\ 
k_0' - k_0 + 1, &\mbox{if } \theta_0 \leq \theta < \theta_0',\\
k_0' - k_0 - 1, &\mbox{if } \theta_0' \leq \theta < \theta_0,\\
k_0' - k_0, &\mbox{if } \theta < \theta_0,\theta'_0.
\end{cases}
\end{equation}
Let $x < x',$ and let $Z_\theta:= \{ k \in \Z \mid x \leq \theta + 2 \pi k <x' \}.$ It is obvious that for any integer $k$ satisfying $k_0 < k < k_0'$ we have $k \in Z_\theta.$ There are precisely $k_0' - k_0 - 1$ of such $k$ in total. Moreover, we have
\begin{align*}
&k_0 \in Z_\theta \mbox{ if and only if } \theta_0 \leq \theta, \\
&k'_0 \in Z_\theta \mbox{ if and only if } \theta < \theta'_0.
\end{align*}
It follows that \cref{Equation2: Infinitesimal Spectral Flow} and $[\theta;x,x'] = \# Z_\theta$   agree to each other, provided that  $x < x'.$ On the other hand, if $x' < x,$ then $[\theta;x,x'] = -[\theta;x',x] = -(\eta_\theta(1;x) - \eta_\theta(1;x'))/2\pi.$ The formula \cref{Equation1: Infinitesimal Spectral Flow} has been verified. With this result in mind \cref{Equation: Mu-invariant of Ttheta} becomes
\[
\mu(\{T_\theta(t)\}_{t \in [0,1]}) 
= \sum^\infty_{j=1} \frac{\eta_\theta(1;x_j(1)) - \eta_\theta(1;x_j(0))}{2\pi}
= \sum^\infty_{j=1} [\theta;x_j(0),x_j(1)]. 
\]
\end{proof}

\begin{proof}[Proof of \cref{Theorem: Properties of Mu-invariant}]
The well-definedness of $\mu\left(-,\{S(t)\}_{t \in [0,1]} \right)$ follows from \cref{Lemma: Finite Sum} and \cref{Lemma: Mu-invariant for Paths}. The first three assertions (i),(ii),(iii) follow from the formula \cref{Equation: Definition of mu-invariant for Paths} and \cref{Theorem: Flow in Circle Case}. Finally, the last assertion (iv) follows from \cref{Equation: Definition of mu-invariant for Loops}.
\end{proof}

\subsection{Spectral flow}
\cref{Theorem: Big Continuous Enumeration} and \cref{Theorem: Properties of Mu-invariant} allow us to state the following definition;
\begin{defn}
\label{Definition: Spectral Flow}
Let $\cH$ be a separable Hilbert space, and let $\Phi$ be a regular symmetric norm. Let $\eG = \eU_\Phi(\cH,1)$ or $\eG = \eG_\infty(\cH,1).$ Given a continuous one-parameter family $\{T(t)\}_{t \in [0,1]}$ in $\eG,$ its \textbi{spectral flow} is a  function $\flow(-,\{T(t)\}_{t \in [0,1]}) : (0,2\pi) \to \Z$ defined by the formula \cref{Equation: Spectral Flow of T}, where $\{\sigma(T(t))\}_{t \in [0,1]}$ is always viewed as a one-parameter family in $\eS_\infty(\C^*,1)$ through the canonical identification in \cref{Remark: Identification}.
\end{defn}

\begin{thm}
With the notation introduced in \cref{Definition: Spectral Flow} in mind, we have the following assertions:
\begin{enumerate}[(i)]
  \item If $S,T$ are path-homotopic in $\eG,$ then $\flow(-,\{S(t)\}_{t \in [0,1]}) = \flow(-,\{T(t)\}_{t \in [0,1]}).$
  \item If $S(1) = T(0),$ then $\flow(-,\{(S \divideontimes T)(t)\}_{t \in [0,1]}) = \flow(-,\{S(t)\}_{t \in [0,1]}) + \flow(-,\{T(t)\}_{t \in [0,1]}).$
  \item We have $\flow(-,\{S^{-1}(t)\}_{t \in [0,1]}) = - \flow(-,\{S(t)\}_{t \in [0,1]}).$
  \item If $S(0)=S(1)=1,$ then $\flow(-,\{S(t)\}_{t \in [0,1]}) = \mathrm{const}.$
\end{enumerate}
\end{thm}
\begin{proof}
The claim immediately follows from \cref{Theorem: Properties of Mu-invariant}. 
\end{proof}

\appendix
\renewcommand{\thesection}{\Alph{section}} 

\section{Separability and Completeness}
\label{Section: Appendix}

The ultimate purpose of the current section is to prove that the $\Phi$-multiset functor $\Lip \to \Lip$ in \cref{Section: Multiset functors} preserves separability and completeness; 

\begin{thm}
\label{Theorem: Separability and Completeness}
Let $\Phi$ be a regular symmetric norm, and let $(X,x_0)$ be a based metric space. Then the following assertions hold true:
\begin{enumerate}[(i)]
\item If $X$ is separable, then so is $\eS_\Phi(X,x_0).$
\item If $X$ is complete, then so is $\eS_\Phi(X,x_0).$
\end{enumerate}
\end{thm}

\subsection{Notation}

Let $\Phi$ be a regular symmetric norm throughout the current section. Assume that $(\xi_i)_{i}$ is either a finite-sequence of non-negative numbers or an infinite sequence of non-negative numbers converging to $0.$ We define the sequence $\xi^\downarrow = (\xi^\downarrow_i)_{i}$ as the non-increasing rearrangement of $\xi_1, \xi_2,\dots.$ That is, we define $\xi^\downarrow$ through
\begin{equation}
\label{Definition: Non-increasing Rearrangement}
\xi^\downarrow_1
    = \max_{i \in \N} \xi_i, \qquad \xi^\downarrow_1 + \xi^\downarrow_2
    = \max_{i \neq j}\left(\xi_i + \xi_j\right), \qquad \dots
\end{equation}

Let $\ell_\Phi(\R_+)$ be the set of all infinite $\Phi$-summable sequences of non-negative real numbers. Since $\ell_\Phi(\R_+)$ is a closed subset of the Banach space $\ell_\Phi(\R),$ it is a complete metric space. Given a sequence $\xi \in \ell_\Phi(\R_+),$ we let for each $n \in \N$
\begin{align}
\label{Equation: Supscript Truncation} 
	\xi^{[n]}   &:= (\xi_1,\dots,\xi_n,0,0,\dots), \\
\label{Equation: Subscript Truncation} 
	\xi_{[n]}   &:= (\xi_{n+1},\xi_{n+2},\dots).
\end{align}
Let $(Y,y_0) := (\R_+,0),$ and let $\rho(x,y) := |x - y|$ for each $x,y \in \R_+.$ Since $d(x_0,-) : X \to \R_+$ is a $1$-Lipschitz continuous mapping, it induces the $1$-Lipschitz continuous mapping $d(x_0,-)_* : \eS_\Phi(X,x_0) \to \eS_\Phi(\R_+,0)$ as in \cref{Section: Multiset functors}. For brevity we let
\[
d_* := d(x_0,-)_*.
\]
With this convention in mind we have $d_*(O_{x_0}) = \{0, 0, \dots\}^* = O_0.$ Note that the induced mapping $d_*$ preserves rank;
\begin{equation}
\label{Equation: Rank Preserving Mapping}
\rank d_*(S) = \rank S, \qquad S \in  \eS_\Phi(X,x_0).
\end{equation}

\subsection{Continuity of the non-increasing rearrangement}
We shall prove that $\ell_\Phi(\R_+) \ni \xi \longmapsto \xi^\downarrow \in \ell_\Phi(\R_+)$ is $1$-Lipschitz continuous (\cref{Theorem: Majorisation is Continuous}). Let us start with the following finite-dimensional version;

\begin{lem}
\label{Lemma1: Symmetric Norms}
If $\xi,\eta \in \R^n$ are finite $n$-tuples of non-negative numbers, then
\[
\Phi(|\xi^\downarrow_1 - \eta^\downarrow_1|, \dots, |\xi^\downarrow_n -
\eta^\downarrow_n|,0,0,\dots) \leq \Phi(|\xi_1 - \eta_1|, \dots, |\xi_n - \eta_n|,0,0,\dots)
\]
\end{lem}
\begin{proof}
Given finite $n$-tuples $x,y \in \R^n$ of non-negative numbers, we say that $x$ is \textit{weakly majorised} by $y,$ if
\begin{equation}
\label{Equation: Weak Majorisation}
\sum^k_{i=1} x^\downarrow_i \leq \sum^k_{i=1} y^\downarrow_i, \qquad k = 1, \dots, n.
\end{equation} 
It is shown in \cite[Theorem 6.A.2.a]{Book:Marshall-Olkin79} that $x := (|\xi^\downarrow_i - \eta^\downarrow_i|)_{i \in \N}$ is weakly majorised by $y := (|\xi_i - \eta_i|)_{i \in \N}.$ The claim follows from \cref{Lemma: Majorisation}.
\end{proof}

\begin{lem}
\label{Lemma2: Symmetric Norms}
If $\xi \in  \ell_\Phi(\R_+),$ then $\left(\xi^{[n]}\right)^\downarrow \to \xi^\downarrow$ as $n \to \infty.$
\end{lem}
Note that $\left(\xi^{[n]}\right)^\downarrow \neq \left(\xi^\downarrow\right)^{[n]}$ in general (otherwise this claim would be trivial).
\begin{proof}
We may assume without loss of generality that all of the terms of $\xi \in \ell_\Phi(\R_+)$ are non-zero. It follows from the regularity of $\Phi$ that for any $\epsilon > 0$ there exists an integer $n_0,$ such that $\Phi\left(\xi_{[n_0]}\right) < \epsilon / 2$ and $\Phi\left({\xi^\downarrow}_{[n_0]}\right) < \epsilon / 2.$ Furthermore, there exists an integer $N > n_0,$ such that for all $n > N$ we have $\Phi\left(\xi_{[n]}\right) < \xi^\downarrow_{n_0}.$ It follows that for all $n > N$ the numbers $\xi_{n+1}, \xi_{n+2}, \dots$ are all strictly less than $\xi^\downarrow_{n_0}$ by \cref{Corollary: Basic Properties of Symmetric Norms} (iii). That is, the first $n_0$ terms of $\xi^\downarrow, \left(\xi^{[n]}\right)^\downarrow$ must be identical. It follows that
\[
\Phi\left(\xi^\downarrow - \left(\xi^{[n]}\right)^\downarrow \right) 
= \Phi\left({\xi^\downarrow}_{[n_0]} -  {\left(\xi^{[n]} \right)^\downarrow}_{[n_0]} \right) 
< \frac{\epsilon}{2} + \Phi\left({\left(\xi^{[n]}\right)^\downarrow}_{[n_0]}\right), \qquad n > N.
\]
It remains to prove $\Phi\left({\left(\xi^{[n]}\right)^\downarrow}_{[n_0]}\right) < \epsilon/2$ for any $n > N.$ For such $n,$ there exists a permutation $\pi$ of $\{1,\dots,n\},$ such that 
$
\left(\xi^{[n]}\right)^\downarrow = (\xi_{\pi_1}, \dots, \xi_{\pi_n},0,0,\dots).
$
That is, $\xi_{\pi_1} \geq \dots \geq \xi_{\pi_n}.$ It follows that
\begin{align*}
\Phi\left({\left(\xi^{[n]}\right)^\downarrow}_{[n_0]}\right)
    &= \Phi(\xi_{\pi_{n_0 + 1}}, \dots, \xi_{\pi_{n}},0,0,\dots) \\
    &= \Phi(\xi_{\pi_{n}}, \dots, \xi_{\pi_{n_0 + 1}},0,0,\dots) \\
    &\leq \Phi(\xi_{n_0 + 1}, \dots, \xi_{n},0,0,\dots) \\
    &\leq \Phi\left(\xi_{[n_0]}\right) < \frac{\epsilon}{2},
\end{align*}
where the second equality follows from the invariance of $\Phi$ under permutations, and the first inequality follows from \cref{Corollary: Basic Properties of Symmetric Norms} with in mind that $\xi_{\pi_{n}} \leq \dots \leq \xi_{\pi_{n_0 + 1}}$ are the first $n - n_0$ numbers taken from the non-decreasing rearrangement of $\xi_1, \dots, \xi_n.$ The proof is complete.
\end{proof}

\begin{pro}
\label{Theorem: Majorisation is Continuous}
The mapping $\ell_\Phi(\R_+) \ni \xi \longmapsto \xi^\downarrow \in \ell_\Phi(\R_+)$ is $1$-Lipschitz continuous;
\begin{equation}
\label{Majorization Inequality}
\Phi(|\xi^\downarrow_1 - \eta^\downarrow_1|, |\xi^\downarrow_2 -
\eta^\downarrow_2|,\dots) \leq \Phi(|\xi_1 - \eta_1|, |\xi_2 - \eta_2|, \dots), 
\qquad \xi,\eta \in \ell_\Phi(\R_+).
\end{equation}
\end{pro}
\begin{proof}
It follows from \cref{Lemma1: Symmetric Norms} that $\Phi\left(\left| \left(\xi^{[n]}\right)^\downarrow - \left(\eta^{[n]}\right)^\downarrow \right|\right) \leq \Phi\left(\left|{\xi^{[n]}} - {\eta^{[n]}}\right|\right)$ for all $n \in \N.$ By \cref{Lemma2: Symmetric Norms}, taking the limit as $n \to \infty$ completes the proof.
\end{proof}

\subsection{Finite-rank multisets}

For each $k = 0, 1, 2, \dots,$ let $\eF_k(X,x_0)$ be the set of all multisubsets of $(X,x_0)$ with rank less than or equal to $k.$ Let $\eF_\infty(X,x_0) := \bigcup_{k \in \N} \eF_k(X,x_0).$ Note that $\eF_k(X,x_0) \subseteq \eS_\Phi(X,x_0)$ for each $k=0,1,\dots, \infty.$

\begin{lem}
\label{Lemma: Dense Lemma}
We have the following assertions:
\begin{enumerate}[(i)]
\item The set $\eF_\infty(X,x_0)$ is a dense subset of $\eS_\Phi(X,x_0).$
\item The set $\eF_k(X,x_0)$ is a closed subset of $\eS_\Phi(X,x_0)$ for each finite $k=0,1,\dots.$
\end{enumerate}
\end{lem}
\begin{proof}
For the first assertion, if $S = \{s_1,s_2,\dots\}^*$ belongs to $\eS_\Phi(X,x_0),$ then
\[
\lim_{i \to \infty} d_\Phi(S, \{s_1,\dots,s_i\}^*) \leq \lim_{i \to \infty} \Phi(0,\dots,0,d(x_0,s_{i+1}),d(x_0,s_{i+2}),\dots) = 0,
\]
where the last equality follows from \cref{Lemma: Metric Space Remark} (iii). The claim follows.

As for the second assertion, it suffices to prove the claim for $(X,x_0) = (\R_+,0)$ by the rank preserving property \cref{Equation: Rank Preserving Mapping}. Assume that there exists a sequence $(S_n)_{n \in \N}$ in $\eF_k(\R_+,0)$ converging to $S_0 \in \eS_\Phi(\R_+,0).$ Note that for each $n = 0, 1, 2,\dots,$ the multiset $S_n$ admits the following unique non-increasing enumeration;
\begin{equation}
\label{Equation1: Decreasing Enumeration}
S_n = \{s^n_1, s^n_2, \dots\}^*, \qquad s^n_1 \geq s^n_2 \geq \dots.
\end{equation}
It follows from \cref{Theorem: Majorisation is Continuous} that
\begin{equation}
\label{Equation2: Decreasing Enumeration}
d_\Phi(S_m, S_n) \geq \Phi(s^m_1-s^n_1, s^m_2-s^n_2, \dots) \geq \sup_{i \in \N} |s^m_i - s^n_i|, \qquad m,n \in \N \cup \{0\}.
\end{equation}
It follows that for each $i \in \N$ we $s^n_i \to s^0_i$ as $n \to \infty.$ Since $\rank S_n \leq k$ for each $n \in \N,$ we have $s^n_{k+i} = 0$ for each $i = 1,2, \dots,$ and so $s^n_{k+i} \to s^0_{k+i} = 0$ as $n \to \infty.$  We get $S_0 \in \eF_k(\R_+,0).$
\end{proof}

\subsection{Separability (Proof of \texorpdfstring{\cref{Theorem: Separability and Completeness} (i)}{})}

\begin{proof}[Proof of \cref{Theorem: Separability and Completeness} (i)]
By \cref{Lemma: Dense Lemma} (i), it suffices to construct a countable dense subset of $\eF_\infty(X,x_0).$ Since $X$ is separable, it has a countable dense subset $A.$ We may assume without loss of generality that $x_0 \in A.$ We show that the countable set $\{S \in \eF_\infty(X,x_0) \mid \supp S \subseteq A \}$ is a dense subset of $\eF_\infty(X,x_0)$ (recall that the set of all infinite sequences of natural numbers which are eventually constant is countable). Let $S =
\{s_1, \dots, s_n\}^*$ be in $\eF_\infty(X,x_0).$ Since $A$ is a dense subset of $X,$ there exist $n$ sequences $(s^1_i)_{i \in \N}, \dots, (s^n_i)_{i \in \N}$
in $A$ converging to $s_1, \dots, s_n$ respectively. It follows from \cref{Equation: Finite Sum} that $\{s^1_i, \dots, s^n_i\}^* \to S$ as $i \to \infty.$ The claim follows.
\end{proof}

\subsection{Completeness (Proof of \texorpdfstring{\cref{Theorem: Separability and Completeness} (ii)}{})}

We shall assume that $X$ is a complete metric space throughout. Let us first prove the following special case of \cref{Theorem: Separability and Completeness} (ii);

\begin{lem}
\label{Lemma: Baby Completeness}
The metric space $\eS_\Phi(\R_+,0)$ is complete.
\end{lem}
\begin{proof}
Let $(S_n)_{n \in \N}$ be a Cauchy sequence in $\eS_\Phi(\R_+,0).$ We assume that each $S_n$ admits the unique non-increasing enumeration $s^n = (s^n_j)_{j \in \N}$ given by \cref{Equation1: Decreasing Enumeration}. It follows from \cref{Equation2: Decreasing Enumeration} that the sequence $(s^n)_{n \in \N}$ is a Cauchy sequence in the complete metric space $\ell_\Phi(\R_+),$ and so it converges to $s^0 = (s^0_1, s^0_2,\dots)$ in $\ell_\Phi(\R_+).$ Note that $S_0 := \{s^0_1, s^0_2,\dots\}^*$ is $\Phi$-summable, since $s^0 \in \ell_\Phi(\R_+).$ Now,
\begin{align*}
d_\Phi(S_n, S_0) \leq \Phi(|s^n_1 - s^0_1|, |s^n_2 - s^0_2|,\dots) = \Phi(s^n - s^0) \to 0, \qquad n \to \infty.
\end{align*}
That is, the Cauchy sequence $(S_n)_{n \in \N}$ converges to $S_0.$ The claim follows.
\end{proof}

Evidently, if $(S_n)_{n \in \N}$ is a Cauchy sequence in  $\eS_\Phi(X,x_0),$ then $(d_*(S_n))_{n \in \N}$ is a Cauchy sequence in the complete metric space $\eS_\Phi(\R_+,0).$ It follows from \cref{Lemma: Baby Completeness} that the sequence $(d_*(S_n))_{n \in \N}$ always converges to some multiset in $\eS_\Phi(\R_+,0).$ With this fact in mind, we shall prove the following sequence version of \cref{Proposition: Reducing Open Sets};

\begin{lem}
\label{Lemma: Split Cauchy Sequence into Finitely Many Parts}
Let $(S_n)_{n \in \N}$ be a Cauchy sequence in $\eS_\Phi(X,x_0)$ with $D := \lim_{n \to \infty} d_*(S_n),$ and let $(I_0,\dots, I_k)$ be a positively separated tuple of open subsets of $\R_+$ in the sense of \cref{Definition: Definition of Positively Separated Tuples} (ii), such that
$
D \subseteq I_0 \cup \dots \cup I_k.
$
Let $U_0, \dots, U_k$ be the inverse images of $I_0,\dots, I_k$ under $d(x_0,-) : X \to \R_+.$ Then the sequences $(S^0_n)_{n \in \N}, \dots,(S^k_n)_{n \in \N}$ given by the following formula are Cauchy sequences;
\begin{equation}
\label{Equation: Split Cauchy Sequence into Finitely Many Pars}
S^i_n := S_n \cap U_i, \qquad i=0,1, \dots, k.
\end{equation}
Furthermore, there exists an integer $N \in \N$ such that
\begin{enumerate}[(i)]
\item We have $S_n  = S^0_n + \dots + S^k_n$ for all $n \geq N.$
\item We have $d_\Phi(S_m^i, S_n^i) \leq d_\Phi(S_m,S_n)$ for each $i = 0, \dots, k$ and for each $m,n \geq N.$ 
\item We have $\rank S_m^i = \rank S_n^i$ for each $i = 1, \dots, n$ and for each $m,n \geq N.$
\end{enumerate}
\end{lem}
Note that we assume $0 \in I_0$ as in \cref{Definition: Definition of Positively Separated Tuples} (ii), so that $x_0 \in U_0.$ Moreover,
\begin{equation}
\label{Equation: Uj and Ij}
d_*(S \cap U_j) = d_*(S) \cap I_j, \qquad S \in \eS_\Phi(X,x_0), \qquad j= 0, \dots, k.
\end{equation}
\begin{proof}
Given the positively separated tuple $(I_0,\dots, I_k)$ as above, let us first show that $(U_0, \dots, U_k)$ is positively separated tuple of open subsets of $X.$ Indeed, if $i \neq j,$ then
\[
\dist (U_i,U_j)
    = \inf_{(u_i,u_j) \in U_i \times U_j} d(u_i,u_j)
    \geq \inf_{(u_i,u_j) \in U_i \times U_j} |d(u_i,x_0) - d(x_0,u_j)|
    \geq \dist (I_i,I_j).
\]
Let $I := I_0 \cup \dots \cup I_k,$ and let $U := U_0 \cup \dots \cup U_k.$ Since $d_*(S_n) \to D$ as $n \to \infty,$ there exists an integer $N$ such that for all $n \geq N$ we have $d_*(S_n) \in \eS^{I}_\Phi(\R_+, 0)$ by \cref{Lemma: Openness of Restricted Sp}. It follows that $S_n \in \eS^U_\Phi(X,x_0)$ for all $n \geq N.$ We can increase $N,$ if necessary, to ensure that for each $m,n \geq N$ the metric $d_\Phi(S_m,S_n)$ never exceeds the separation of the tuple $(U_0, \dots, U_k).$ The claim now follows from \cref{Lemma: Intersection Inequality}.
\end{proof}

\begin{cor}
\label{Lemma: Constant-rank Cauchy Sequence}
Let $\Phi$ be a regular symmetric norm, and let $(X,x_0)$ be a based complete metric space. Then $\eF_k(X,x_0)$ is a complete metric space for each finite $k= 0, 1,2, \dots.$
\end{cor}
\begin{proof}
Let $(S_n)_{n \in \N}$ be a Cauchy sequence in $\eF_k(X,x_0).$ We shall proceed by induction on $k.$ Since the set $\eF_0(X,x_0)$ consists only of one multiset $O_{x_0},$ we shall start with the base step $k=1.$ Then there exists a sequence $(s_n)_{n \in \N}$ of points in $X,$ such that $S_n = \{s_n\}^*$ for all $n \in \N.$ It follows from \cref{Equation: Phi-estimate} that $(s_n)_{n \in \N}$ is Cauchy sequence in $X,$ and so $(s_n)_{n \in \N}$ converges to some point $s_0 \in X.$ It follows from \cref{Equation: Finite Sum} that $(S_n)_{n \in \N}$ converges to $S_0 := \{s_0\}^*.$

For the induction step, assume that the claim has been proved for $k$ replaced by any smaller number. It suffices to consider the non-trivial case that $d_*(S_n) \to D$ for some $D$ which is not $O_0 = \{0,0,\dots\}^*.$ It follows from \cref{Lemma: Dense Lemma} (ii) that $0 < \rank D \leq k.$ We can then construct a positively separated tuple $(I_0, \dots, I_k),$ such that $D \subseteq I_0 \cup \dots \cup I_k,$ where for each $i=0, \dots, k$ the multiset $D \cap I_i$ has rank strictly less than $k.$ For such $i,$ we define $S^i := (S^i_n)_{n \in \N}$ according to \cref{Equation: Split Cauchy Sequence into Finitely Many Pars}. The claim follows from \cref{Lemma: Split Cauchy Sequence into Finitely Many Parts} and the induction hypothesis.
\end{proof}

We are now in a position to prove the completeness of $\eS_\Phi(X,x_0);$
\begin{proof}[Proof of \cref{Theorem: Separability and Completeness} (ii)]
Let $(S_n)_{n \in \N}$ be a Cauchy sequence in $\eS_\Phi(X,x_0),$ and let $D := \lim_{n \to \infty} d(S_n).$ Suppose that
$
\supp D = \{d_1, d_2, \dots, 0\},
$
where $d_1 > d_2 > \dots > 0,$ and that each $d_i$ has  multiplicity $m_i$ in the multiset $D.$ We shall consider the non-trivial case $\rank D = \infty.$ Let $\{I_i\}_{i \in \N} = \{(\alpha_i,\beta_i)\}_{i \in \N}$ be a countable family of open intervals in
$\R,$ such that $
\bigcap_{i \in \N} [\alpha_i, \beta_i] = \emptyset
$ and $d_i \in I_i$ for each $i \in \N.$

(A) For any $k \in \N,$ we set $I_0 := [0, \beta_{k+1}),$ so that $(I_0,\dots,I_k)$ forms a positively separated tuple of open subsets
of $\R_+,$ and so \cref{Lemma: Split Cauchy Sequence into Finitely Many Parts} holds true. We define $k + 1$ sequences $(S^0_n)_{n \in \N}, \dots, (S^k_n)_{n \in \N}$ by \cref{Equation: Split Cauchy Sequence into Finitely Many Pars}. It follows from \cref{Lemma: Constant-rank Cauchy Sequence} that the last $k$ sequences $(S^1_n)_{n \in \N}, \dots, (S^k_n)_{n \in \N}$ all converge to finite-rank multisets $S^1_{0},\dots,S^k_0 \in \eS_\Phi(X,x_0).$ It follows that for each $i=1, \dots,k,$ we have
\begin{align}
\label{Equation1: Separability Theorem}
&d_*(S^i_0)
    = \lim_{n \to \infty} d_*(S_n \cap U_i)
    = \lim_{n \to \infty} \left(d_*(S_n) \cap I_i \right)
    = D \cap I_i, \\
\label{Equation2: Separability Theorem}
&S^i_0 \subseteq U_i, 
\end{align}
where the second equality in \cref{Equation1: Separability Theorem} follows from \cref{Equation: Uj and Ij} and the last equality in \cref{Equation1: Separability Theorem} follows from \cref{Corolary: Intersection Continuity} and \cref{Lemma: Openness of Restricted Sp}. It follows from \cref{Equation: Rank Preserving Mapping} that the rank of each $S^i_0$ is $m_i.$ That is, each $S^i_0$ admits a representation
$
S^i_0 = \{s^i_1, \dots, s^i_{m_i}\}^*,
$
so that $d_*(S^i_0) = D \cap I_i = \{d_i, \dots, d_i\}^*$ for each $i = 1, \dots, k.$

(B) The previous part allows us to define $S_0 := \{s^1_1,\dots,s^1_{m_1}, s^2_2,\dots,s^2_{m_2}, \dots\}^*,$ the $\Phi$-summability of which follows from that of $d_*(S_0) = D.$ We show that $S_n \to S_0$ as $n \to \infty.$ Let $\epsilon > 0$ be arbitrary. Then there exists large enough $k \in \N,$ such that
\begin{align}
\label{Equation3: Separability Theorem}
& I_0 := [0,\beta_{k+1}), \\
\label{Equation4: Separability Theorem}
&\rho_\Phi(D \cap I_0, O_0)
    = \Phi(d_{k+1}, \dots, d_{k+1}, d_{k+2}, \dots, d_{k+2},\dots)
< \frac{\epsilon}{4},
\end{align}
where $\rho(x,y) = |x - y|$ is the standard metric on $\R_+$ and each $d_{k+j}$ in \cref{Equation4: Separability Theorem} is repeated $m_{k+j}$ times. Since the last equality in \cref{Equation1: Separability Theorem} also holds true for $i = 0,$ there exists $N \in \N,$ such that 
\begin{equation}
\label{Equation5: Separability Theorem}
\rho_\Phi(d_*(S_n) \cap I_0, O_0) 
\leq \rho_\Phi(d_*(S_n) \cap I_0, D \cap I_0)  + \rho_\Phi(D \cap I_0, O_0) 
< \frac{\epsilon}{2}, \qquad n \geq N.
\end{equation}
It follows from \cref{Lemma: Split Cauchy Sequence into Finitely Many Parts} that we can always increase $N,$ if necessary, so that
\begin{equation}
\label{Equation6: Separability Theorem}
S_n = S^0_n + \dots + S^k_n, \qquad 
\sum^k_{i=1} d_\Phi(S^i_n, S^i_0) < \frac{\epsilon}{4}, 
\qquad n \geq N.
\end{equation}
Note that \cref{Equation2: Separability Theorem} holds for each $i = k+1, k+2, \dots,$ and so $S^0_0 := S_0 - (S^1_0 + \dots + S^k_0) = S_0 \cap U_0.$ It follows that for all $n \geq N$ we have the following estimate;
\begin{align*}
d_\Phi(S_n, S_0)
&= d_\Phi(S^0_n + \dots + S^k_n, S^0_0 + \dots + S^k_0) \\
&\leq d_\Phi(S^0_n, S^0_0) + \sum^k_{i=1} d_\Phi(S^i_n, S^i_0)\\
&\leq d_\Phi(S^0_n,O_{x_0}) + d_\Phi(O_{x_0}, S^0_0) + \sum^k_{i=1} d_\Phi(S^i_n, S^i_0) \\
&= \rho_\Phi(d_*(S^0_n),O_{0}) + \rho_\Phi(d_*(S^0_0),O_{0}) + \sum^k_{i=1} d_\Phi(S^i_n, S^i_0) \\
&= \rho_\Phi(d_*(S_n) \cap I_0, O_0) + \rho_\Phi(D \cap I_0, O_0) +\sum^k_{i=1} d_\Phi(S^i_n, S^i_0),
\end{align*}
where the first inequality follows from \cref{Sum Inequality} and the last equality follows from \cref{Equation: Uj and Ij}. It follows from \crefrange{Equation4: Separability Theorem}{Equation5: Separability Theorem}
that $d_\Phi(S_n, S_0) < \epsilon$ for each $n \geq N,$ and so $S_n \to S_0$ as $n \to \infty.$ The claim follows.
\end{proof}




\begin{thebibliography}{99}

\bibitem[APS75]{Atiyah-Patodi-Singer75}
M.~F. Atiyah, V.~K. Patodi, and I.~M. Singer.
\newblock Spectral asymmetry and {R}iemannian geometry. {I}.
\newblock {\em Math. Proc. Cambridge Philos. Soc.}, 77:43--69, 1975.

\bibitem[AD19]{Azamov-Daniels19}
N.~Azamov and T.~Daniels.
\newblock Singular spectral shift function for resolvent comparable operators.
\newblock {\em Math. Nachr.}, 292(9):1911--1930, 2019.


\bibitem[Aza11]{Azamov11}
N.~Azamov.
\newblock Absolutely continuous and singular spectral shift functions.
\newblock {\em Dissertationes Math.}, 480:102, 2011.

\bibitem[BD99]{Bhatia-Davis99}
R.~Bhatia and C.~Davis.
\newblock Perturbation of extended enumerations of eigenvalues.
\newblock {\em Acta Sci. Math. (Szeged)}, 65(1-2):277--286, 1999.

\bibitem[Bha97]{Book:Bhatia97}
R.~Bhatia.
\newblock {\em Matrix analysis}, volume 169 of {\em Graduate Texts in
  Mathematics}.
\newblock Springer-Verlag, New York, 1997.

\bibitem[BS88]{Bhatia-Sinha88}
R.~Bhatia and K.~B. Sinha.
\newblock A unitary analogue of {K}ato's theorem on variation of discrete
  spectra.
\newblock {\em Lett. Math. Phys.}, 15(3):201--204, 1988.

\bibitem[DS88]{Book:Dunford-Schwartz88}
N.~Dunford and J.~T. Schwartz.
\newblock {\em Linear operators. Part II. Spectral Theory}.
\newblock Wiley-Interscience, 1988.

\bibitem[GK69]{Book:Gohberg-Krein69}
I.~C. Gohberg and M.~G. Kre{\u\i}n.
\newblock {\em Introduction to the theory of linear nonselfadjoint operators}.
\newblock Translations of Mathematical Monographs. American Mathematical
  Society, Providence, R.I., 1969.

\bibitem[GM00]{Gesztesy-Makarov00}
F.~Gesztesy and K.~A. Makarov.
\newblock The $\xi$ operator and its relation to krein's spectral shift
  function.
\newblock {\em J. Anal. Math.}, 81(1):139--183, 2000.

\bibitem[Hat02]{Book:Hatcher02}
A.~Hatcher.
\newblock {\em Algebraic topology}.
\newblock Cambridge University Press, Cambridge, 2002.

\bibitem[Kat87]{Kato87}
T.~Kato.
\newblock Variation of discrete spectra.
\newblock {\em Comm. Math. Phys.}, 111(3):501--504, 1987.

\bibitem[Kat95]{Book:Kato95}
T.~Kato.
\newblock {\em Perturbation theory for linear operators}.
\newblock Classics in Mathematics. Springer-Verlag, Berlin, 1995.
\newblock Reprint of the 1980 edition.

\bibitem[Kre53]{Krein53}
M.~G. Kre{\u\i}n.
\newblock On the trace formula in perturbation theory.
\newblock {\em Mat. Sbornik N.S.}, 33(75):597--626, 1953.

\bibitem[Lif52]{Lifshitz52}
I.~M. Lifshitz.
\newblock On a problem of the theory of perturbations connected with quantum
  statistics.
\newblock {\em Uspehi Matem. Nauk (N.S.)}, 7(1(47)):171--180, 1952.

\bibitem[MO79]{Book:Marshall-Olkin79}
A.W. Marshall and I.~Olkin.
\newblock {\em Inequalities: theory of majorization and its applications}.
\newblock Mathematics in science and engineering. Academic Press, 1979.

\bibitem[Pus01]{Pushnitski01}
A.~Pushnitski.
\newblock The spectral shift function and the invariance principle.
\newblock {\em J. Funct. Anal.}, 183(2):269--320, 2001.

\bibitem[RS78]{Book:ReedSimon4}
M.~Reed and B.~Simon.
\newblock {\em Methods of Modern Mathematical Physics: Vol.: 4. : Analysis of
  Operators}.
\newblock Academic Press, 1978.
\newblock International Series in Pure and Applied Mathematics.

\bibitem[RS95]{Robbin-Salamon95}
J.~Robbin and D.~Salamon.
\newblock The spectral flow and the {M}aslov index.
\newblock {\em Bull. London Math. Soc.}, 27(1):1--33, 1995.

\bibitem[Rud76]{Book:Rudin76}
W.~Rudin.
\newblock {\em Principles of mathematical analysis}.
\newblock McGraw-Hill Book Co., New York, third edition, 1976.
\newblock International Series in Pure and Applied Mathematics.

\bibitem[Sim98]{Simon98}
B.~Simon.
\newblock Spectral averaging and the krein spectral shift.
\newblock {\em Proc. Amer. Math. Soc.}, 126(5):1409--1413, 1998.

\bibitem[Sim05]{Book:Simon05}
B.~Simon.
\newblock {\em Trace ideals and their applications}, volume 120 of {\em
  Mathematical Surveys and Monographs}.
\newblock American Mathematical Society, Providence, RI, second edition, 2005.

\bibitem[Tan14]{Tanaka14}
Y.~Tanaka.
\newblock A topological approach to spectral flow.
\newblock Master's thesis, School of Computer Science, Engineering and
  Mathematics, Flinders University, 2014.

\bibitem[Yaf92]{Book:Yafaev92}
D.~R. Yafaev.
\newblock {\em Mathematical scattering theory}, volume 105 of {\em Translations
  of Mathematical Monographs}.
\newblock American Mathematical Society, Providence, RI, 1992.
\newblock General theory.

\end{thebibliography}
\end{document}